\documentclass[11pt]{amsart}

\usepackage{amsfonts,amscd}
\usepackage{amssymb}
\usepackage{url}

\usepackage{textcomp}
\usepackage[english]{babel}

\allowdisplaybreaks

\theoremstyle{plain}
\newtheorem{theorem}                {Theorem}      [section]
\newtheorem{proposition}  [theorem]  {Proposition}
\newtheorem{corollary}    [theorem]  {Corollary}
\newtheorem{lemma}        [theorem]  {Lemma}

\theoremstyle{definition}
\newtheorem{example}      [theorem]  {Example}
\newtheorem{remark}       [theorem]  {Remark}
\newtheorem{definition}   [theorem]  {Definition}

\numberwithin{equation}{section}

\def \R{{\mathbb R}}

\def \s {{\mathbb S}}
\def \h {{\mathbb H}}

\usepackage{color}

\def \link {~}
\def \1 {\`}

\title[Polyharmonic curves]{Polyharmonic curves in semi-Riemannian manifolds}

\author{S.~Montaldo}
\address{Universit\`a degli Studi di Cagliari\\
Dipartimento di Matematica e Informatica\\
Via Ospedale 72\\
09124 Cagliari, Italia}
\email{montaldo@unica.it}

\author{A.~Ratto}
\address{Universit\`a degli Studi di Cagliari\\
Dipartimento di Matematica e Informatica\\
Via Ospedale 72\\
09124 Cagliari, Italia}
\email{rattoa@unica.it}

\author{A.~Sanna}
\address{Universit\`a degli Studi di Cagliari\\
Dipartimento di Matematica e Informatica\\
Via Ospedale 72\\
09124 Cagliari, Italia}
\email{antonio.sanna4@unica.it}

\usepackage{hyperref}
\begin{document}
\begin{abstract}
Let $(M^m_t,g)$ be a semi-Riemannian manifold of dimension
$m$ with a non-degenerate metric of \textit{index} $t$, $m\geq 2$, $1 \leq t \leq m-1$. The main aim of this paper is to investigate the existence of Frenet curves in $(M^m_t,g)$ which are polyharmonic of order $r$, shortly, $r$-harmonic. We shall focus primarily on the cases that the ambient space is a semi-Riemannian space form $N^m_t(c)$ of sectional curvature $c$, a ruled Lorentzian surface or a suitable, possibly warped, product space. 
We shall obtain existence, non-existence and classification results. 
\end{abstract}

\subjclass[2010]{Primary: 58E20. Secondary: 53C50.}

\keywords{Polyharmonic curves, Frenet curves, helices, semi-Riemannian space forms, Lorentzian surfaces, Robertson-Walker space time}

\thanks{The authors are members of the Italian National Group G.N.S.A.G.A. of INdAM. The work was partially supported by the Project {ISI-HOMOS} funded by Fondazione di Sardegna and by BC\_TRITECH - Bando a cascata (PE00000015 - B83C22004800006).  The author A. S. was supported by a NRRP scholarship -  funded by the European Union - NextGenerationEU - Mission 4, Component 1, 
Investment 3.4.}
\maketitle
\section{Introduction and statement of the results}\label{Sec-Intro}

Let $(M^m_t,g)$ be an $m$-dimensional semi-Riemannian (or, pseudo-Riemannian) manifold, where $g$ is a non-degenerate metric of index $t$, $1\leq t \leq m-1$. If $\gamma: I\to M^m_t$ is a smooth curve defined on an open real interval $I$, then we say that $\gamma$ has one of the three causal characters, i.e., \textit{space-like}, \textit{light-like} (\textit{null}) or \textit{time-like} if $g\left (\gamma',\gamma'\right ) >0$, $g\left (\gamma',\gamma'\right )=0$ or $g\left (\gamma',\gamma'\right )<0$ respectively on $I$.

In this paper we shall study curves $\gamma$ which are either space-like or time-like. Then, to avoid trivialities arising from re-parametrization, we shall always assume that $\gamma$ is parametrized with respect to the arc length $s$, i.e.,
\[
\langle \gamma'(s),\gamma' (s)\rangle=g(\gamma'(s),\gamma' (s))=\epsilon_1\,,
\]
where $\epsilon_1=1$ if $\gamma$ is space-like and $\epsilon_1=-1$ if $\gamma$ is time-like. 
From now on, we shall use the notation $T=\gamma'(s)$. 

The main topic of this paper is the study of polyharmonic curves of order $r$, shortly \textit{$r$-harmonic curves}, $r \geq 2$. 
We recall that $\gamma$ is $r$-harmonic if and only if its $r$-tension field $\tau_r(\gamma)$ vanishes. This condition, up to an irrelevant sign $(-1)^{r-1}$, is expressed by (see \cite{MR4542687,MR2869168,MR4308322, MR1, MR2})
\begin{equation}\label{r-harmonicity-curves}
\tau_r(\gamma)=\nabla^{2r-1}_T T+ \sum_{\ell=0}^{r-2}(-1)^\ell R\left (\nabla^{2r-3-\ell}_T T ,\nabla^{\ell}_T T\right )T=0 \,,
\end{equation}
where $\nabla^{0}_T T=T,\,\nabla^{k}_T T=\nabla_T \left ( \nabla^{k-1}_T T\right ) $ and $R$ is the semi-Riemannian curvature operator of the ambient space.

We point out that any geodesic is trivially $r$-harmonic for all $r \geq 1$. Therefore, we say that $\gamma$ is a \textit{proper} $r$-harmonic curve if it is $r$-harmonic and \textit{not} a geodesic.

Several papers on the biharmonic case in this context can be found in the literature. By way of example, we cite \cite{MR1609044,Du-Zhang, MR4142001,MR4393014,MR1990566,MR2251644,MR3879834,MR4160739}. More precisely, the classification of proper biharmonic curves in the flat $3$-dimensional Lorentz-Minkowski space was achieved in \cite{Du-Zhang,MR4142001}. Other related results for biharmonic curves in non flat Lorentz space forms were obtained by Sasahara in \cite{Sasahara}. Further complementary results were proved in \cite{MR3879834,MR4160739} in the special case that the ambient space is the $3$-dimensional Minkowski space. 

An interesting feature of these results is the fact that, differently from the Riemannian case, proper biharmonic curves do exist also in the \textit{flat} Minkowski space.

By contrast, in the literature we found no paper dealing with $r$-harmonic curves in a semi-Riemannian manifold when $r \geq3$. 

Thus, our aim is to start the investigation of this open case. 

We shall restrict our attention to the the study of the so-called \textit{$n$-Frenet curves}. We shall follow the approach of \cite{Song, Ziplar}, where the authors focus on the study of helices of order $n\geq 3$.
\begin{definition}[\textbf{Frenet curves (helices) of order $n$}]\label{def-helices-order-n} Let $(M^m_t,\langle, \rangle)$ be an $m$-dimensional semi-Riemannian manifold of index $t$, $1 \leq t \leq m-1$. Let $\gamma:I \to M^m_t$ be a \textit{nonnull} curve parametrized by the arc length $s$. We say that $\gamma$ is a \textit{Frenet curve of order $n$}, $n \leq m$, if it admits a \textit{Frenet frame field} $\left \{F_1,\ldots,F_n \right \}$ along $\gamma$ which verifies:
\begin{equation}\label{Frenet-field-general-pseudo}
\begin{cases}
\nabla_T F_1= \epsilon_2 \,k_1\,F_2 \\
\nabla_T F_i=-\epsilon_{i-1}\,k_{i-1} F_{i-1}+\,\epsilon_{i+1}\,k_i F_{i+1} \qquad (1<i<n)\\
\nabla_T F_n=-\epsilon_{n-1}\,k_{n-1} F_{n-1} \,, 
\end{cases} 
\end{equation}
where $F_1=T$ and 
\[
\epsilon_{j}=\langle F_j,F_j \rangle(=\pm 1) \quad (1\leq j \leq n)\quad {\rm and}\quad  \quad \langle F_i,F_j \rangle=0 \quad (i\neq j) \,,
\]
and the functions $k_1(s), \ldots, k_{n-1}(s)$ are positive. The function $k_i(s)$ is called the $i$-th \textit{curvature} of the curve.

In the special case that $M^m_t$ is oriented and $n=m\geq 2$ we allow that $k_{n-1}(s)$ can assume any real value because in this case $F_n$ can be defined by requiring that $\{F_1,\ldots,F_n \}$ is a positively oriented orthonormal base of $T_{\gamma(s)}M^m_t$. Also, we say that an $n$-Frenet curve $\gamma$ is \textit{full} if $k_{n-1}(s)$ never vanishes.
Finally, we say that an $n$-Frenet curve is a \textit{helix} if its curvatures $k_1(s), \ldots, k_{n-1}(s)$ are constant functions which will be denoted $\kappa_1, \ldots, \kappa_{n-1}$.
\end{definition}

\begin{remark}\label{remark-not-skew}
System~\ref{Frenet-field-general-pseudo} can be written in matrix form as follows:

\[
\begin{bmatrix}
\nabla_T F_1\\
\nabla_T F_2\\
\vdots\\
\nabla_T F_n
\end{bmatrix}
=
\Omega
\begin{bmatrix}
 F_1\\
 F_2\\
\vdots\\
 F_n
\end{bmatrix}
\]

where
{\small
\[
\Omega=\begin{bmatrix}
0&\epsilon_2 k_1&0&0&0&\cdots&0&0&0\\
-\epsilon_1 k_1&0&\epsilon_3 k_2&0&0&\cdots&0&0&0\\
0&-\epsilon_2 k_2&0&\epsilon_4 k_3&0&\cdots&0&0&0\\
\vdots&\vdots&\vdots&\vdots&\vdots&\vdots&\vdots&\vdots&\vdots\\
0&0&0&0&\cdots&0&-\epsilon_{n-2} k_{n-2}&0&\epsilon_{n} k_{n-1}\\
0&0&0&0&\cdots&0&0&-\epsilon_{n-1} k_{n-1}&0\\
\end{bmatrix}
\]
}\vspace{2mm}

We point out that the matrix $\Omega$ is skew-symmetric if and only if
\[
\epsilon_1=\epsilon_2=\cdots=\epsilon_n
\]
which essentially corresponds to the Riemannian case.
\end{remark}

\begin{remark} In the notation of Definition~\ref{def-helices-order-n}, we must have $0\leq \#\{\epsilon_j \colon \epsilon_j=-1 \}\leq t$ and $0\leq \#\{\epsilon_j \colon \epsilon_j=1 \}\leq (m-t)$.
\end{remark}

First, let us discuss $2$-Frenet curves. In this context, we simply write $k(s)$ to denote the first and unique curvature $k_1(s)$. For biharmonic curves ($r=2$) into a $2$-dimensional Lorentzian space form $N^2_1(c)$ we have the following comprehensive and well-known characterization (see, for example, \cite{Sasahara}).
\begin{theorem}\label{Cor-biharm-2-dim-space-form}
Let $\gamma$ be a space-like proper biharmonic curve parametrized by the arc length $s$ into a $2$-dimensional Lorentzian space form $N^2_1(c)$. Then $k(s)$ is a constant, say $k(s)=\kappa>0$. Moreover, biharmonicity is equivalent to
\begin{equation}\label{biharm-cond-2-dim-space-form}
\kappa^2 + c =0\,.
\end{equation}
In particular, if $c\geq0$, then there exists no space-like proper biharmonic curve in $N^2_1(c)$.
\end{theorem}
In the general case of $r$-harmonic curves ($r>3$) it is not automatic that the curvature $k(s)$ is constant. Thus, as a first, preliminary result, we shall prove:
\begin{theorem}\label{Prop-r-harm-2-dim-space-form}
Let $r \geq 2$. Let $\gamma$ be a space-like curve parametrized by the arc length $s$ into a $2$-dimensional Lorentzian space form $N^2_1(c)$. Let us assume that $\gamma$ has constant, non-zero curvature $k(s)$, say $k(s)=\kappa>0$. Then $\gamma$ is proper $r$-harmonic if and only if
\begin{equation}\label{r-harm-cond-2-dim-space-form}
\kappa^2 + (r-1)c =0\,.
\end{equation}
In particular, there exists no space-like proper $r$-harmonic curve when $c \geq 0$. 
\end{theorem}
\begin{remark} It is interesting to compare Theorem~\ref{Prop-r-harm-2-dim-space-form} with the corresponding result for the Riemannian case. 
Indeed, we recall that  (see \cite{MR4542687} and \cite{MR3007953}) it was proved that the version of \eqref{r-harm-cond-2-dim-space-form} in the case of curves into a $2$-dimensional  Riemannian space form $N^2(c)$ is
\[
\kappa^2 - (r-1)c =0\,.
\]
\end{remark}

Things do change if we consider $2$-Frenet helices in a higher dimensional semi-Riemannian space form. The main reason which explains the appearance of significant differences is the fact that in this case $\epsilon_1$ and $\epsilon_2$ may have the same sign. 
Then, in this context, the corresponding version of Theorem~\ref{Prop-r-harm-2-dim-space-form} is:
\begin{theorem}\label{Prop-r-harm-m-dim-space-form}
Let $r \geq 2$. Let $\gamma$ be a $2$-Frenet helix into an $m$-dimensional semi-Riemannian space form $N^m_t(c)$, $m \geq 3$, $1 \leq t \leq m-1$. Then $\gamma$ is proper $r$-harmonic if and only if
\begin{equation}\label{r-harm-cond-2-dim-space-form-bis}
\kappa^2 -\epsilon_2 \, (r-1)c =0\,.
\end{equation}
\end{theorem}

The case of triharmonic curves appears to be special. Indeed, unlike the case where $r>3$, $3$-harmonicity forces $k(s)$ to be a constant. More precisely, we prove:
\begin{theorem}\label{Prop-3-harmon+2-Frenet-implica-k-costante} Let $\gamma$ be a triharmonic $2$-Frenet curve into an $m$-dimensional semi-Riemannian space form $N^m_t(c)$, $m \geq 2$, $1 \leq t \leq m-1$. Then $k(s)$ is a constant function.
\end{theorem}
This result leads us to the classification of proper triharmonic $2$-Frenet curves in semi-Riemannian space forms. Specifically, we shall prove the following:
\begin{corollary}\label{Th-classif-triharm-space-form}
Let $\gamma$ be a proper triharmonic $2$-Frenet curve in a semi-Riemannian space form $N^m_t(c)$, $m \geq 2$, $1 \leq t \leq m-1$. Then necessarily $\epsilon_2 \,c>0$, where $\epsilon_2=\langle N,N\rangle$, and the curvature of $\gamma$ is a constant given by
\[
\kappa= \sqrt{2 \,\epsilon_2\,c}\,.
\]
\end{corollary}
\begin{remark}
A $2$-Frenet helix as in Corollary~\ref{Th-classif-triharm-space-form} necessarily lies on a totally geodesic  $N^2_{t'}(c)$, $0 \leq t' \leq 2$. Since $r$-harmonicity is not affected by multiplication of the ambient metric by $-1$, there are only two relevant cases, i.e., $t'=0$ and $t'=1$. If $t'=0$, then $\epsilon_1=\epsilon_2=1$, $c>0$ and the curve is a triharmonic circle in an Euclidean sphere. The geometrically interesting case is $t'=1$. In this case, $\epsilon_1\,\epsilon_2=-1$ and, considering the natural embeddings of $\s^2_1(c)$ in $\R^3_1$ or $\h^2_1(c)$ in $\R^3_2$, the curve in both cases is the intersection of $N^2_1(c)$ with a plane.
\end{remark}
\vspace{2mm}

In Theorem~\ref{Prop-3-harmon+2-Frenet-implica-k-costante} we have used in an essential way the fact that the \( 2 \)-Frenet curves lie in an \( m \)-dimensional semi-Riemannian manifold with constant sectional curvature. 

Therefore, a natural problem is to investigate the existence of \( r \)-harmonic, \( n \)-Frenet curves with \textit{non constant} curvatures in a semi-Riemannian space such that its sectional curvature is not constant. 

In this direction, let \( \gamma: I \to \R^3_1 \) be a curve parametrized by the arc length \( s \) that satisfies the Frenet equations.

\begin{equation}\label{eq-Frenet-R3,1}
\begin{cases}
\nabla_T T = \epsilon_2 \,k_\gamma(s) \,N \\
\nabla_T N = -\epsilon_1\,k_\gamma(s) \,T + \epsilon_3 \, \tau_\gamma(s) B \\
\nabla_T B = - \epsilon_2 \,\tau_\gamma(s) \,N \,.
\end{cases} 
\end{equation}

Then we associate to $\gamma$ a ruled surface $S_\gamma$ immersed in $\R^3_1$ and locally described by the parametrization
\begin{equation}
X(s,v)=\gamma(s)+v\,N(s) \,,
\end{equation}
where $s\in I$ and $v \in (-\delta, \delta)$ for some $\delta>0$.

Our main result in this context is a semi-Riemannian version of a result of \cite{MR4308322}:
\begin{theorem}\label{Th-trihar-curve-k-nonconstant}
There exist a curve $\gamma: I \to \R^3_1$ parametrized by the arc length $s$, which verifies the Frenet equations \eqref{eq-Frenet-R3,1} with $\epsilon_3=1$, and a small $\delta>0$ such that the associated $S_\gamma$ is a Lorentz surface and $\gamma(s)=X(s,0)$ is a proper triharmonic curve in $S_\gamma$ with non constant curvature $k(s)$.
\end{theorem}

Now, we turn our attention to the study of $3$-Frenet curves, where the two relevant curvatures will be denoted $k(s),\,\tau(s)$ as it is usual in the literature. In this setting, our first result is:

\begin{theorem}\label{Th-r-harm-cond-3-dim-space-form}
Let $\gamma$ be a $3$-Frenet helix parametrized by the arc length $s$ into an $m$-dimensional semi-Riemannian space form $N^m_t(c)$, $m \geq 3$, $1 \leq t \leq m-1$. Then 
\begin{itemize}
\item[{\rm (i)}] $\gamma$ is a proper biharmonic curve if and only if $\epsilon_2 \, (\epsilon_1 \, \kappa^2 + \epsilon_3 \tau^2) = c \, \epsilon_1$;
\item[{\rm (ii)}] $\gamma$ is a proper triharmonic curve if and only if $\epsilon_2 \, (\epsilon_1 \, \kappa^2 + \epsilon_3 \tau^2 )^2 = c \, \big( \epsilon_1\,\epsilon_3 \, \tau^2\, + 2 \kappa^2 \big)$;
\item[{\rm (iii)}] $\gamma$ is a proper $r$-harmonic curve, $r \geq  4$, if and only if the following condition holds:
\begin{equation}\label{eq-r-harmonicity-spaceforms-r>3}
 \big(\epsilon_1 \kappa^2 + \epsilon_3 \tau^2\big)^{r-3}\,\Big [\epsilon_2 \, (\epsilon_1 \, \kappa^2 + \epsilon_3 \, \tau^2)^2 - c \, \big( \epsilon_1\,\epsilon_3 \, \tau^2\, + (r - 1) \kappa^2 \big )\Big ]=0\,.
\end{equation}
\end{itemize}
\end{theorem}
According to Theorem~\ref{Th-r-harm-cond-3-dim-space-form} the existence of proper $r$-harmonic $3$-Frenet helices depends strongly on the sign of $\epsilon_1,\epsilon_2,\epsilon_3$ and $c$. Therefore, we discuss now some relevant cases. 

First, we focus on the geometrically significant Lorentzian case $N^3_1(c)$. 
In this space necessarily we have $\epsilon_3 = -\epsilon_1 \, \epsilon_2$. Our first corollary shows that in the flat case the condition of $r$-harmonicity depends on the casual character of the normal vector field $N$. More precisely, we have
\begin{corollary}\label{Cor-facile}
Let $\gamma$ be a $3$-Frenet helix parametrized by the arc length $s$ into $\R^3_1$. 
Then $\gamma$ is proper $r$-harmonic, $r \geq 2$, if and only if its normal vector field $N$ is space-like and $\kappa^2 = \tau^2$.
\end{corollary}
\begin{remark}
A $3$-Frenet helix in $\R^3_1$ with $\epsilon_2=1$ and verifying $\kappa^2 = \tau^2$ is proper $r$-harmonic for all $r\geq 2$. This property is consistent with the fact that the ambient space is flat. Explicit examples of curves which verify the condition $\kappa^2 = \tau^2$ are given in \cite[Proposition~3.4]{Sasahara}.
\end{remark}
Next, we examine biharmonicity in a non flat Lorentzian case. Then we obtain:
\begin{corollary}
Let $\gamma$ be a $3$-Frenet helix parametrized by the arc length $s$ into a $3$-dimensional Lorentzian space form $N^3_1(c)$ with $c \neq 0$.
\begin{itemize}
\item[{\rm (i)}] If $c<0$ and the normal vector field $N$ is space-like, then $\gamma$ is a proper biharmonic curve if and only if $\tau^2 -\kappa^2 =-c$.
\item[{\rm (ii)}]  If $c<0$ and the normal vector field $N$ is time-like, then $\gamma$ is a proper biharmonic curve if and only if $\tau^2 +\kappa^2 =-c$. 
\item[{\rm (iii)}] If $c>0$ and the normal vector field $N$ is space-like, then $\gamma$ is a proper biharmonic curve if and only if $\kappa^2 =\tau^2 +c$.
\item [{\rm (iv)}] If $c>0$ and the normal vector field $N$ is time-like, then $\gamma$ cannot be proper biharmonic.
\end{itemize}
\end{corollary}
Finally, here is a consequence of Theorem~\ref{Th-r-harm-cond-3-dim-space-form} in the non flat case. 
\begin{corollary}\label{Cor-tecnico}
Assume that $r \geq 3$. Let $\gamma$ be a $3$-Frenet helix parametrized by the arc length $s$ into a non flat semi-Riemannian space form $N^{3}_1(c)$. 
\begin{itemize}
\item[(i)] If $c>0$ and the normal vector field $N$ is space-like, then $\gamma$ is a proper $r$-harmonic curve if and only if $\kappa^2 = \tau^2$ for $r > 3$ or, for $r \geq 3$, if one of the following cases holds:
\vspace{0.6 mm}
\begin{itemize}
\item[(a)] $\kappa^2 \geq c(r-1)$ and $\tau^2 = \dfrac{2\kappa^2 - c \pm \sqrt{c^2 + 4c(r-2)\kappa^2}}{2}\,;$
\item[(b)] $\kappa^2 < c(r-1)$ and $\tau^2 = \dfrac{2\kappa^2 - c + \sqrt{c^2 + 4c(r-2)\kappa^2}}{2}\,.$
\end{itemize}
\vspace{0.6 mm}
\item[(ii)] If $c>0$ and the normal vector field $N$ is time-like, then $\gamma$ cannot be proper $r$-harmonic  for any $r \geq 3$.
\vspace{0.6 mm}
\item[(iii)] If $c<0$ and the normal vector field $N$ is space-like, the $\gamma$ is a proper $r$-harmonic curve if and only if $\kappa^2 = \tau^2$ for $r > 3$ or
$$\kappa^2 \leq - \dfrac{c}{4 (r-2)} \,\,\,{\rm and} \,\,\,\tau^2 = \dfrac{2\kappa^2 - c \pm \sqrt{c^2 + 4c(r-2)\kappa^2}}{2}$$
for $r \geq 3 $.
\vspace{0.6 mm}
\item[(iv)] If $c<0$ and the normal vector field $N$ is time-like, then $\gamma$ is a proper $r$-harmonic curve, $r \geq 3$, if and only if :
\vspace{0.6 mm}

 $\kappa^2 \leq -c(r-1)$ and $\tau^2 = \dfrac{-2\kappa^2 - c + \sqrt{c^2 - 4c(r-2)\kappa^2}}{2}\,.$
\end{itemize}
\end{corollary}

\begin{remark}
It is interesting to point out that Corollary~\ref{Cor-tecnico} guarantees the existence of \( r \)-harmonic helices in a semi-Riemannian space form \( N^{3}_1(c) \) with \( c < 0 \), regardless of the character of the normal vector field $N$. This should be compared with the corresponding Riemannian case, where proper $r$-harmonic helices do not exist when the ambient space has constant negative curvature.
\end{remark}

Next, let us turn our attention to $n$-Frenet curves with $n \geq 4$. Then our first result is:

\begin{theorem}\label{biharmonic-n-Frenet-curve}
Let $\gamma$ be an $n$-Frenet curve parametrized by the arc length $s$ into an $m$-dimensional semi-Riemannian manifold $(M_t^m, g)$, $1 \leq t \leq m-1$, $m \geq n \geq 4$. If
\begin{equation}\label{eq:Rorthogonalfi}
\langle R(F_1,F_2)F_1,W \rangle =0 \quad \forall\, W \in \big({\rm Span}\{F_1,\ldots,F_n\}\big)^\perp \,,
\end{equation}
then $\gamma$ is proper biharmonic if and only if:
\begin{equation}\label{eq-biharm-n-Frenet}
\begin{cases}
k_1 =\kappa_1= {\rm constant} > 0\\
\epsilon_1 \, k_1^2 + \epsilon_3 \, k_2^2 = \langle R(F_2, F_1) F_1, F_2 \rangle\\
k_2' = - \langle R(F_2, F_1) F_1 , F_3 \rangle\\
\epsilon_3 \, k_2 \, k_3 = - \langle R(F_2, F_1) F_1 , F_4 \rangle\\
\langle R(F_2, F_1) F_1 , F_i \rangle = 0 \quad 5 \leq i \leq n.
\end{cases} 
\end{equation}
\end{theorem}
Since \eqref{eq:Rorthogonalfi} is automatically satisfied when the ambient space is a semi-Riemannian space form, we have the following immediate consequence:
\begin{corollary}\label{Cor-biharmonic-n-Frenet-space-forms}
Let $\gamma$ be an $n$-Frenet curve parametrized by the arc length $s$ into an $m$-dimensional semi-Riemannian space form $N^m_t(c)$, $1 \leq t \leq m-1$, $m\geq n \geq 4$. Then $\gamma$ is proper biharmonic if and only if:
\begin{equation}\label{eq-biharm-n-Frenet-space-forms}
\begin{cases}
k_1 =\kappa_1= {\rm constant} >0\\
k_2 =\kappa_2= {\rm constant} >0\\
\epsilon_1 \, k_1^2 + \epsilon_3 \, k_2^2 = c \,\epsilon_1\,\epsilon_2\,\\
\epsilon_3 \,k_2 \,k_3=0\,.
\end{cases} 
\end{equation}
In particular, there exists no proper biharmonic, full $n$-Frenet curve in $N^m_t(c)$, $m\geq n \geq 4$.
\end{corollary}
\begin{remark} The last sentence of Corollary~\ref{Cor-biharmonic-n-Frenet-space-forms} displays a phenomenon analogous to the Riemannian case (compare with \cite[Theorem~1.9]{MR2}). Moreover, as in the Riemannian case, $3$-Frenet biharmonic curves in a semi-Riemannian space-form are necessarily helices. 

Therefore, condition (i) of Theorem~\ref{Th-r-harm-cond-3-dim-space-form} provides the complete classification of proper biharmonic Frenet curves in a semi-Riemannian space form.
\end{remark}
Specializing to the triharmonic case we obtain:
\begin{theorem}\label{cor-n-Frenet-tri-spaceforms}
Let $\gamma$ be a $n$-Frenet helix parametrized by the arc length $s$ into $N_t^n(c)$, $1 \leq t \leq n-1$.
\begin{itemize} 
\item[(i)] If $n=4$, then $\gamma$ is triharmonic if and only if:
\begin{eqnarray*}
 (\epsilon_1\, \kappa_1^2 \,  + \epsilon_3 \, \kappa_2^2)^2  + \epsilon_2 \, \epsilon_4  \, \kappa_2^2 \, \kappa_3^2 & = &c \epsilon_1 \epsilon_2 \, (2 \epsilon_1\, \kappa_1^2 \,  + \epsilon_3 \, \kappa_2^2 ) \\
 \epsilon_2 (\epsilon_1\, \kappa_1^2 \,  + \epsilon_3 \, \kappa_2^2 )+ \epsilon_3 \epsilon_4 \,  \kappa_3^2   & = & c \epsilon_1 \,.
\end{eqnarray*}
\item[(ii)] If $n=5$, then $\gamma$ is triharmonic if and only if:
\begin{eqnarray*}
 (\epsilon_1\, \kappa_1^2 \,  + \epsilon_3 \, \kappa_2^2)^2  + \epsilon_2 \, \epsilon_4  \, \kappa_2^2 \, \kappa_3^2 & = &c \epsilon_1 \epsilon_2 \, (2 \epsilon_1\, \kappa_1^2 \,  + \epsilon_3 \, \kappa_2^2 ) \\
 \epsilon_2 (\epsilon_1\, \kappa_1^2 \,  + \epsilon_3 \, \kappa_2^2 )+ \epsilon_4(\epsilon_3 \,   \kappa_3^2 + \epsilon_5 \kappa_4^2)  & = & c \epsilon_1 \,.
\end{eqnarray*}
\end{itemize}
\end{theorem}

\begin{remark}
The Riemannian version of Theorem~\ref{cor-n-Frenet-tri-spaceforms} was first obtained by Maeta in \cite[Proposition~5.9]{MR3007953}. Maeta's result corresponds to the condition of Theorem~\ref{cor-n-Frenet-tri-spaceforms} with $\epsilon_j=1$ for all $j$.
\end{remark}
As a consequence of Theorem~\ref{cor-n-Frenet-tri-spaceforms} we are able to point out \textit{an important  difference with respect to the Riemannian case}. More precisely, we know (see \cite{MR3007953} and \cite[Theorem~5.1]{MR2}) that there exists no proper triharmonic $n$-Frenet helix in a Riemannian space form $N^n(c)$, $n \geq 4$. By contrast, as an application of Theorem~\ref{cor-n-Frenet-tri-spaceforms}, we have: 
\begin{example}\label{cor-controes-triharmonic}
There exists a space-like proper triharmonic $5$-Frenet helix in $\s^5_1$ with the following features:
\[
\epsilon_3=-1,\, \epsilon_j=1 \quad {\rm if}\,\,j\neq 3 \,;\quad \quad
\kappa_1= \kappa_2=\kappa_3=1,\,\, \kappa_4=\sqrt{2}\, \,.
\]
\end{example}
Also, with a similar analysis:
\begin{example}\label{cor-controes-triharmonic-bis}
There exists a space-like proper triharmonic $4$-Frenet helix in $\s^4_2$ with the following features:
\[
\epsilon_2=\epsilon_3=-1,\,\,\, \epsilon_1=\epsilon_4=1  \,;\quad \quad
\kappa_1=\sqrt{2},\,\, \kappa_2=2\,,\,\, \kappa_3=1\, \,.
\]
\end{example}
\vspace{3mm}

Now, we turn our attention to more general ambient spaces.
First, we establish the following result:
\begin{proposition}\label{Th-general-r-harmonicity-3-Frenet-helices}
Let $\gamma$ be a $3$-Frenet helix in any semi-Riemannian manifold $M^m_t$, where $m\geq 3$ and $1\leq t \leq m-1$. If the curvatures $\kappa, \tau$ of the helix $\gamma$ verify
\begin{equation}\label{eq-r-harm-general-helices-r>=4}
\epsilon_1 \kappa^2+\epsilon_3 \tau^2=0\,,
\end{equation}
then $\gamma$ is proper $r$-harmonic for all $r\geq 4$.
\end{proposition}
To continue our investigation in ambient spaces with non constant sectional curvature we focus on the Lorentzian product space $\left ( \R \times N^{m-1}(c),g_{{\rm prod}}\right )$, where $\left (N^{m-1}(c),g \right )$ is a \textit{Riemannian} $(m-1)$-dimensional space form with $c \neq 0$, $m\geq 4$ and 
$$
g_{{\rm prod}}=-dt^2+g\,.
$$ 
In this context, we introduce a family of curves which will play an interesting role. More precisely, let $\alpha(s)$ be a $3$-Frenet helix in $N^{m-1}(c)$. Note that here $s$ is the arc length parameter in $N^{m-1}(c)$. Then we consider curves $\gamma$ in $\left ( \R \times N^{m-1}(c),g_{{\rm prod}}\right )$ of the type
\begin{equation}\label{eq-def-gamma}
\gamma(s)=(d_1\, s, \alpha (d_2 s))
\end{equation}
where, to insure that $\gamma$ is parametrized by the arc length, we require that the constants $d_1,d_2$ satisfy
\begin{equation}\label{eq-d}
\epsilon_1+d_1^2 >0 \quad {\rm and} \quad d_2=\sqrt{\epsilon_1+d_1^2}\,.
\end{equation}
Our main result in this setting is:
\begin{theorem}\label{Th-r-harm-eliche-product-result1}
Assume $c\neq 0, m \geq 4$. Let $\gamma$ be a $3$-Frenet curve in $\left ( \R \times N^{m-1}(c),g_{{\rm prod}}\right )$ of the type \eqref{eq-def-gamma} and assume that
\begin{equation}\label{eq-d-condizione}
\epsilon_3 \left (\tau_\alpha^2-\epsilon_1 \,d_1^2\,\kappa_\alpha^2 \right )>0\,,
\end{equation}
where $k_\alpha,\tau_\alpha$ are the curvatures of the helix $\alpha$ in $N^{m-1}(c)$.
Then $\gamma$ is a helix which does not verify \eqref{eq-r-harm-general-helices-r>=4}.
Moreover, for all $r\geq2$, $\gamma$ is proper $r$-harmonic in $\left ( \R \times N^{m-1}(c),g_{{\rm prod}}\right )$ if and only if $\alpha$ is proper $r$-harmonic in $N^{m-1}(c)$.
\end{theorem}
As an immediate consequence of Theorem~\ref{Th-r-harm-eliche-product-result1} we have the following explicit construction of $r$-harmonic $3$-Frenet curves in $\left ( \R \times N^{m-1}(c),g_{{\rm prod}}\right )$.
\begin{corollary}
Let $\alpha(s)$ be a proper $r$-harmonic $3$-Frenet helix in $N^{m-1}(c)$, $r \geq 2$. Then, choosing $d_1,d_2\neq 0$ such that $d_1^2=d_2^2+1$, the curve $\gamma(s)$ defined in \eqref{eq-def-gamma} is a time-like, proper $r$-harmonic helix in $\left ( \R \times N^{m-1}(c),g_{{\rm prod}}\right )$.
\end{corollary}

In the last section of the paper we shall analyse another important case where the sectional curvature of the ambient space is not constant. More precisely, let $J$ be a real open interval. Then a \textit{Robertson-Walker space-time} is a semi-Riemannian manifold $\mathcal{RW}^{m}_1 =\left ( J \times N^{m-1}(c),g_f\right )$, where  $g_f$ is a Lorentzian metric  defined by
\begin{equation}\label{eq-gf-walker}
g_f = -dt^2 + f^2(t) \, g\,,
\end{equation}
where $f$ is a smooth positive function on $J$. These space-times, which are largely used in theoretical physics and in the study of cosmological phenomena, appeared for the first time in the context of harmonic maps in \cite{MR1035340}. 

Our main results in this context are:
\begin{theorem}\label{biharmonic-cond-Robertson-Walker}
Let $\alpha(s), s \in I$, be a geodesic in $N^{m-1}(c)$ and consider a curve $\gamma : I \to \mathcal{RW}^{m}_1$, $\gamma(s) = (t_0, \alpha(d\,s))$, where $t_0$ is any fixed real value in the interval $J$ and $d=1/f(t_0)$. Assume that $r \geq 2$. Then $\gamma(s)$ is a proper $r$-harmonic curve in $\mathcal{RW}^{m}_1$ if and only if
\[
f'(t_0) \neq 0\qquad{\rm and}\qquad  f'^2(t_0) + (r-1)f(t_0)f''(t_0) = 0\,.
\]

\end{theorem}
\begin{theorem}\label{Th-2-RW}
Let $\gamma : I \to \mathcal{RW}^{m}_1$ be a curve parametrized by the arc length of the form $\gamma(s) = (t_0, \alpha(s))$, where $t_0$ is a fixed value on $J$ such that $f'(t_0)=0$. Then $\gamma(s)$ is $r$-harmonic in $\mathcal{RW}^{m}_1$ if and only if $\beta(s)=\alpha(f(t_0)\,s)$ is $r$-harmonic in $N^{m-1}(c)$.
\end{theorem}
As a final remark, we recall that the quantity
\[
q=- \frac{f(t)f''(t)}{f'^2(t)}
\]
is called the \textit{deceleration parameter} of the Robertson-Walker space-time $\mathcal{RW}^{m}_1$. Relevant cases of physical interest are those corresponding to constant positive values of $q$ which are achieved by taking 
\[
f(t)=t^\lambda\,, \quad 0<\lambda<1
\]
which gives
\[
q=-\frac{\lambda-1}{\lambda}\,.
\] 
Taking $\lambda=(r-1)/r$, as an immediate consequence of Theorem~\ref{biharmonic-cond-Robertson-Walker} we have:
\begin{corollary}\label{Cor-RW} Let us fix $r\geq2$ and consider a Robertson-Walker space-time $\mathcal{RW}^{m}_1$ with constant positive deceleration parameter 
\[
q=\frac{1}{r-1} \quad {\rm i.e.,} \,\, f(t)=t^{\frac{r-1}{r}}\,\,.
\]
Let $\alpha(s), s \in I$, be a geodesic in $N^{m-1}(c)$ and consider any curve $\gamma : I \longrightarrow \mathcal{RW}^{m}_1$ of the type $\gamma(s) = (t_0, \alpha(s/f(t_0)))$, where $t_0$ is any positive fixed real value. Then $\gamma$ is proper $r$-harmonic.
\end{corollary}

Our paper is organized as follows. 

In order to make this article reasonably self-contained we have recalled in Section~\ref{Sec-preliminaries} some relevant preliminary notions. In Section~\ref{Sec-2-Frenet} we shall prove the stated results concerning $2$-Frenet curves in space forms.
In Section~\ref{Sec-curva-curvatura-non-costante} we shall prove the existence of a proper triharmonic curve with non constant curvature $k(s)$ in a ruled Lorentzian surface $S_\gamma$, as stated in Theorem~\ref{Th-trihar-curve-k-nonconstant}.
In Section~\ref{Sec-3-Frenet-curves} we shall focus on $3$-Frenet curves and  we shall prove the corresponding results stated in this introduction.
Section~\ref{Sec-n-Frenet} contains the proofs of the stated results concerning $n$-Frenet curves, $n \geq 4$. 
In Section~\ref{sec-Lorentzian-product} we shall prove Proposition~\ref{Th-general-r-harmonicity-3-Frenet-helices} and Theorem~\ref{Th-r-harm-eliche-product-result1}.
Last, in Section~\ref{Sec-Rob-Walker} we shall prove the stated results concerning $r$-harmonic curves in Robertson-Walker space-times. 

\section{Preliminaries}\label{Sec-preliminaries}

A basic reference for semi-Riemannian geometry is the classical book of O'Neill (see \cite{Neill}), but for the specific topics treated in this paper we also refer to \cite{MR4552081, Dong, Liu, MR3198740, Sasahara}.

Let $(M^m_t,g)$ be a semi-Riemannian manifold of dimension
$m$ with a \textit{non-degenerate} metric of \textit{index} $t$. Since in this paper we are not interested in the Riemannian case we shall always assume that $1 \leq t \leq m-1$. We recall that non-degenerate means that the only vector $X \in T_pM$ satisfying $g_p(X, Y ) = 0$ for all $Y \in T_pM$ 
is $X = 0$, for any $p \in M $.  A local \textit{orthonormal} frame field
of $(M^m_t,g)$ is a set of local vector fields $\{e_i\}_{i=1}^m$ such that $g(e_i, e_j) = \varepsilon_i \delta_{ij}$, with
$\varepsilon_1=\ldots \varepsilon_t=-1$, $\varepsilon_{t+1}=\ldots \varepsilon_m=1$.

Note that we shall write
\[
||X||=\sqrt{|g(X,X)|}
\]
for any vector field $X$ tangent to $M$. Also, the notation $\langle X,Y \rangle$ for $g(X,Y)$ will be used throughout this paper.

In this article we shall focus primarily on the case that the ambient space is a semi-Riemannian space form. 
Therefore, we now fix terminology and notations concerning these ambient spaces.

First, the $m$-dimensional pseudo-Euclidean space with index $t$ is denoted by $\R^m_t=(\R^m,\langle,\rangle)$, where 
\[
\langle x,y \rangle= -\sum_{i=1}^t x_i y_i +\sum_{i=t+1}^m x_i y_i \,.
\]
The $m$-dimensional semi-Riemannian sphere, denoted by $\s^m_t (c)$ is defined as follows:
\begin{equation}\label{eq-def-pseudosfere}
\s^m_t (c)= \left \{x \in \R^{m+1}_t \colon \langle x,x \rangle = \frac{1}{c} \right \} \quad \quad (c>0)\,.
\end{equation}
$\s^m_t (c)$, with the induced metric from $\R^{m+1}_t$, is a complete semi-Riemannian manifold with index $t$ and constant positive sectional curvature $c$.

The $m$-dimensional semi-Riemannian hyperbolic space, denoted by $\h ^m_t (c)$, is defined by
\begin{equation}\label{eq-def-pseudohyperbolic}
\h^m_t (c)= \left \{x \in \R^{m+1}_{t+1} \colon \langle x,x \rangle =\frac{1}{c} \right \} \quad \quad (c<0)\,.
\end{equation}
$\h^m_t (c)$, with the induced metric from $\R^{m+1}_{t+1}$, is a complete semi-Riemannian manifold with index $t$ and constant negative sectional curvature $c$.

A semi-Riemannian space form $N^m_t(c)$ refers to one of the three spaces $\R^m_t$, $\s^m_t (c),\h^m_t (c)$. We shall write $\s^m_t $ and $\h^m_t $ for $\s^m_t (1)$ and $\h^m_t (-1)$ respectively.
 
The flat ($c=0$) semi-Riemannian space $\R^m_t$ is called \textit{Minkowski space}, while $\s^m_t (c)$ and $\h^m_t (c)$ are known as \textit{de Sitter space} and \textit{anti-de Sitter space} respectively. When the index is $t=1$, these spaces are also referred to as \textit{Lorentz space forms}. We also point out that $\s^m_t (c)$ is diffeomorphic to $\R^t \times \s^{m-t}$, while $\h^m_t (c)$ is diffeomorphic to $\s^{t} \times \R^{m-t}$. In particular, $\s^m_{m-1} (c)$ and $\h^m_1 (c)$ are not simply connected. 

In this paper we shall adopt the following notation and sign convention for the Riemannian curvature tensor field $R$ of the ambient space $M_t^m$:
\[
\begin{aligned}
 &{R}(X,Y)Z={\nabla}_{X}{\nabla}_{Y}Z
-{\nabla}_{Y}{\nabla}_{X}Z-{\nabla}_{[X,Y]}Z\,, \\
&{R}(X,Y,Z,W)=\langle {R}(X,Y)W,Z \rangle \,.
\end{aligned}
\]
In particular, the sectional curvature tensor field of $N^m_t(c)$ is described by the following simple expression:
\begin{equation}\label{tensor-curvature-N(c)}
R^{N(c)}(X,Y)Z=c\, \big(\langle Y,Z \rangle X-\langle X,Z \rangle Y \big) \quad \quad  \forall \,X,Y,Z \in C(TN(c)) \,.
\end{equation}

In this paper we shall establish several algebraic conditions which are equivalent to the $r$-harmonicity of an $n$-Frenet helix. In the case of curves with non constant curvatures, algebraic conditions will be replaced by a system of differential equations. A common feature to all these cases is the dependence on the $\epsilon_i$'s associated to the Frenet frame. 

Therefore, it is vital for us to know that our results are supported by the following semi-Riemannian version of the Fundamental Theorem of curves. More specifically, any time that we prove that certain curvature conditions associated with appropriate $\epsilon_i$'s verify the $r$-harmonicity equation we shall be able to conclude that there exists a curve with those requisites.

\begin{theorem}[The Fundamental Theorem of Curves in Semi-Riemannian Geometry]\label{Th-Fund-curves}
Let $M=(M^m_t,\langle, \rangle)$ be an $m$-dimensional semi-Riemannian manifold of index $1 \leq t \leq m-1$. Given smooth, positive curvature functions $k_1(s),\ldots,k_{n-1}(s)$ defined on an open real neighbourhood of $s_0\in \R$, an initial point $p_0\in M$ and an initial orthonormal frame $\{T_0,F_{2,0},\ldots,F_{n,0}\}$ at $p_0$ there exists an $n$-Frenet curve $\gamma(s)$ in $M$ parametrized by the arc length $s$ such that
\begin{itemize}
\item [(i)] $\gamma(s_0)=p_0$;
\item [(ii)] The $n$-Frenet frame field $\{T,F_{2},\ldots,F_{n}\}$ of $\gamma$ coincides with 

\noindent $\{T_0,F_{2,0},\ldots,F_{n,0}\}$ at $s=s_0$;
\item [(iii)] The given functions $k_1(s),\ldots,k_{n-1}(s)$ are the curvatures of $\gamma(s)$.
\end{itemize}
\end{theorem}
This result can be obtained adapting the classical argument of the Riemannian case (see \cite{Spivak}) but, since we were not able to find a precise reference in the literature, for the sake of completeness we have added its proof in Appendix~\ref{appendix} at the end of the paper.
\begin{remark} 
The construction of the Frenet frame field $\left \{F_1,\ldots,F_n \right \}$ is based on the orthonormalization of the set of vector fields 
\[
\left \{\nabla_T^0T=T,\nabla_T T,\ldots,\nabla_T^{n-1}T \right \}.
\] 

Here we point out that the correct generalization of the Gram-Schmidt algorithm in a semi-Riemannian context uses the idea of a {\em non-degenerate basis}. Indeed, if $V$ is an $n$-dimensional vector space endowed with a non-degenerate scalar product $g$, an ordered basis $\left \{E_1,\ldots,E_n\right\}$
 for $V$
 is said to be non-degenerate if, for each $k=1,\ldots,n$, the scalar product $g$ restricts to a non-degenerate scalar product on the subspace spanned by $\left \{E_1,\ldots,E_k\right\}$.

Now, if $V$ has an ordered non-degenerate basis $\left \{v_1,\ldots,v_n\right\}$, then the Gram-Schmidt algorithm applied to $\left \{v_1,\ldots,v_n\right\}$ produces an orthonormal basis $\left \{E_1,\ldots,E_n\right\}$ with the property that ${\rm span}\left \{E_1,\ldots,E_k\right\}={\rm span}\left \{v_1,\ldots,v_k\right\}$ for each $k=1,\ldots,n$.
\end{remark}

\section{$r$-harmonic $2$-Frenet curves}\label{Sec-2-Frenet} 
Here we prove the results stated in Section~\ref{Sec-Intro} concerning $2$-Frenet curves in semi-Riemannian space forms.

As a preliminary work, we study \textit{nonnull} $r$-harmonic curves in a semi-Riemannian surface $(M^2_1, \langle,\rangle)$. Let $\gamma:I \to (M^2_1, \langle,\rangle)$ be a \textit{non-degenerate curve}, i.e., a curve such that
\[
\langle \nabla_T T,\nabla_T T\rangle \neq 0
\]
on $I$. Then we can define
\[
\epsilon_2={\rm sign}\left ( \langle \nabla_T T,\nabla_T T\rangle \right)(=\pm 1) 
\]
and the principal unit normal vector field $N$ together with its associated positive curvature function $k(s)$:
\[
\nabla_T T=\epsilon_2 k(s) N \,.
\]
Then $\gamma$ is a $2$-Frenet curve and we have the Frenet system of equations
\[
\begin{cases}
\nabla_T T=\epsilon_2 \,k(s) \,N \\
\nabla_T N=-\epsilon_1\,k(s) \,T
\end{cases}
\]
where $\langle T,T\rangle=\epsilon_1$, $\langle N,N\rangle=\epsilon_2$, $\epsilon_1 \epsilon_2=-1$. 
In this context we also have (see \cite{Neill})
\[
K_M = \frac{\langle R(T, N)N, T \rangle}{\epsilon_1 \, \epsilon_2}= - \langle R(T, N)N, T \rangle\,,
\]
where $K_M$ denotes the Gaussian curvature of the  semi-Riemannian surface $(M^2_1, \langle,\rangle)$. A first routine computation leads us to the following proposition.
\begin{proposition}\label{pro-general-curves-surfaces}
Let $\gamma$ be a non-degenerate curve parametrized by the arc length $s$ into an $m$-dimensional semi-Riemannian manifold $(M_t^m, \langle\, , \,\rangle)$. Then its bitension and tritension fields are given respectively by:

\begin{equation}\label{bitension-2-dim-surface}
\tau_2(\gamma)=\big [-3 \epsilon_1\epsilon_2  k(s) k'(s)\big ]T+\big [ \epsilon_2 \left(K_M(s) \epsilon_1 k(s)+k''(s)-\epsilon_1\epsilon_2 k(s)^3\right)\big ] N\,;
\end{equation}

\begin{equation}\label{tritension-2-dim-sueface}
\begin{aligned}
\tau_3(\gamma)=&\big [-5 \epsilon_1 \epsilon_2 k^{(3)}(s) k(s)+10  k(s)^3 k'(s)-10 \epsilon_1 \epsilon_2 k'(s) k''(s)\big ]T\\
&+\big [\epsilon_2 K_M(s) \big( \epsilon_1 k''(s)-2   \epsilon_2 k(s)^3\big) + \epsilon_2 k^{(4)}(s)-10 \epsilon_1  k(s)^2 k''(s)\\
&-15 \epsilon_1  k(s) k'(s)^2+ \epsilon_2k(s)^5\big ] N\,.
\end{aligned}
\end{equation}
where, with a slight abuse of notation, we have written $K_M(s)$ for $K_M(\gamma(s))$. 
\end{proposition}

Next, let us assume that $k(s)$ is a constant function, say $k(s)=\kappa$. Then it is easy to deduce from the Frenet equations that for all $\ell \geq 0$ we have:
\begin{equation}\label{formula-nablaelleTT-2-Frenet-pseudo}
\begin{cases}
\nabla_T^{2\ell} T= (-\epsilon_1\, \epsilon_2)^{\ell}\,\kappa^{2\ell} \,T \\
\nabla_T^{2\ell+1} T= (-\epsilon_1\, \epsilon_2)^{\ell}\, \epsilon_2\, \kappa^{2\ell+1}\,N.
\end{cases} 
\end{equation}

\begin{remark} When $m=2$, $-\epsilon_1 \epsilon_2 =1$. However, we have preferred to write $-\epsilon_1 \epsilon_2$ in \eqref{formula-nablaelleTT-2-Frenet-pseudo} because this makes the formula applicable in the more general context $m \geq 3 $, as we shall see in the proof of Theorem~\ref{Prop-r-harm-m-dim-space-form}, for instance.
\end{remark}
\begin{proof}[Proof of Theorem~\link\ref{Cor-biharm-2-dim-space-form}]
The result follows immediately from \eqref{bitension-2-dim-surface} with 

\noindent $K_M(s)=c={\rm constant}$. 
\end{proof}

\begin{proof}[Proof of Theorem~\ref{Prop-r-harm-2-dim-space-form}] We have to compute $\tau_r(\gamma)$. Using \eqref{r-harmonicity-curves}, \eqref{tensor-curvature-N(c)} and \eqref{formula-nablaelleTT-2-Frenet-pseudo} it is straightforward to compute
\[
\tau_r(\gamma)= C\, \big [\kappa^2 -\epsilon_2\, (r-1)c \big ] N\,,
\]
where
\[
C=\epsilon_2 \kappa^{2r-3} \,.
\]
\end{proof}

\begin{proof}[Proof of Theorem~\ref{Prop-3-harmon+2-Frenet-implica-k-costante}] The formula \eqref{tritension-2-dim-sueface} with $K_M(s)=c={\rm constant}$ is valid in this context. Explicit integration of the tangential component of the $3$-tension fields leads us to
\begin{equation}\label{eq:k''-esplicita}
k''(s)=\frac{2}{5} \epsilon_1 \epsilon_2 k(s)^3+\frac{c_1}{k(s)^2}
\end{equation}
and
\begin{equation}\label{eq:k'-esplicita}
k'^2(s)= \frac{\epsilon_1 \epsilon_2}{5} k^4(s)-\frac{2\,c_1}{k(s)}+c_2
\end{equation}
where $c_1,c_2$ are integration constants. From \eqref{eq:k''-esplicita} and \eqref{eq:k'-esplicita} it is easy to deduce that 
\[
k^{(4)}(s)=\frac{2 \left(5 \,c_1+2 \epsilon_1 \epsilon_2 k(s)^5\right) \left(-35 \,c_1+15 \,c_2 k(s)+6\epsilon_1 \epsilon_2 k(s)^5\right)}{25 k(s)^5}\,.
\]
Finally, substituting these expressions into the normal component of the $3$-tension field we find that $k(s)$ must be a root of a not identically zero polynomial and so $k(s)$ must be a constant as required.
\end{proof}
Finally, putting together Theorem~\ref{Prop-r-harm-m-dim-space-form} and Theorem~\ref{Prop-3-harmon+2-Frenet-implica-k-costante}, we easily deduce the classification Corollary~\ref{Th-classif-triharm-space-form}. 

\section{Proof of Theorem~\ref{Th-trihar-curve-k-nonconstant}}\label{Sec-curva-curvatura-non-costante}
 
We recall that here the problem is to investigate the existence of $r$-harmonic curves with \textit{non constant} curvatures. To a given curve $\gamma: I \to \R^3_1$, parametri\-zed by the arc length $s$ and verifying the Frenet equations
\begin{equation}\label{eq-Frenet-R3,bis}
\begin{cases}
\nabla_T T = \epsilon_2 \,k_\gamma(s) \,N \\
\nabla_T N =-\epsilon_1\,k_\gamma(s) \,T + \epsilon_3 \, \tau_\gamma(s) B \\
\nabla_T B = - \epsilon_2 \,\tau_\gamma(s) \,N \,,
\end{cases}
\end{equation}
we associate a ruled surface $S_\gamma$ immersed in $\R^3_1$ and locally described by the parametrization
\begin{equation}\label{eq-S-parametrisation}
X(s,v)=\gamma(s)+v\,N(s) \,,
\end{equation}
where $s\in I$ and $v \in (-\delta, \delta)$ for some $\delta>0$.

Now, let us find  under which conditions the curve $\gamma(s)=X(s,0)$ is a triharmonic curve into the ruled surface $S_{\gamma}$ and $S_{\gamma}$ is a Lorentz surface of $\R_1^3$.  
A local frame field for the tangent plane of $S_\gamma$ is given by the coordinate vector fields
\[
X_s =  (1 - \epsilon_1\,k_\gamma(s)v) \, T + \epsilon_3 \, \tau_\gamma(s) v\, B; \hspace{6 mm} X_v = N\,.
\]
Thus the coefficients of the first fundamental form of $S_\gamma$ are
\begin{eqnarray}\label{eq-S-firstFF}\nonumber
 E&=& \langle X_s, X_s \rangle = \epsilon_1 (1 - \epsilon_1\,k_\gamma(s)v)^2 + \epsilon_3 \tau_\gamma^2 v^2\,;\\
F &=& \langle X_s, X_v \rangle = \langle (1 - \epsilon_1\,k_\gamma(s)v) \, T + \epsilon_3 \, \tau_\gamma(s) v\, B, \, N \rangle = 0; \hspace{6 mm} \\\nonumber
G &=& \langle X_v, X_v \rangle = \epsilon_2\,.
\end{eqnarray}
If we denote by $N^S$ the normal vector field to the surface $S_\gamma$ in $\R_1^3$ we have
\[
N^S = \dfrac{X_s \wedge X_v}{\| X_s \wedge X_v \|} = \dfrac{-\epsilon_1 \epsilon_3 \, \tau_\gamma(s) v \, T + \epsilon_3 (1 - \epsilon_1\,k_\gamma(s)v) B}{\sqrt{| \epsilon_1 ( \tau_\gamma(s) v )^2 + \epsilon_3 (1 - \epsilon_1\,k_\gamma(s)v)^2 |}}\,,
\]
where we have used that $T\wedge N= \epsilon_3 B$.
Restricting the unit normal $N^S$ along the curve $\gamma(s) = X(s, 0)$ we obtain $N^S(s) = \epsilon_3 \, B(s)$. Therefore, $S_\gamma$ is a Lorentzian surface provided that $\epsilon_3=1$ and $\delta>0$ is sufficiently small.
 
Next, if we denote by $\nabla^S$  the Levi-Civita connection of $S_\gamma$ and by $\epsilon^{N^S} = \langle N^S, N^S \rangle$, we have
\[
 \nabla_T T = \nabla_T^S T + \epsilon^{N^S} \langle \nabla_T T , N^S \rangle N^S=\nabla_T^S T
 \]
because, from \eqref{eq-Frenet-R3,bis}, $\langle \nabla_T T , N^S \rangle= \langle \epsilon_2 k_\gamma(s) N, \epsilon_3 B \rangle = 0$.
Thus $k(s) = k_\gamma(s)$.

Now we compute the sectional curvature $K_S$ of the surface $S_\gamma$ along $\gamma$. Since $F = 0$, we can use the standard formula  (see \cite[p.81]{Neill})
\[
K_S = -\dfrac{1}{\sqrt{| EG |}} \left( {\rm sign}(E) \left( \dfrac{(\sqrt{|G|})_s}{\sqrt{|E|}} \right)_s+ {\rm sign}(G) \left( \dfrac{(\sqrt{|E|})_v}{\sqrt{|G|}}\right)_v \right)
\]
to obtain
\[
K_S= -\epsilon_2 \left( \dfrac{2\,k_\gamma^2 + 2 \epsilon_1 \epsilon_3 \tau_\gamma^2}{2} - \dfrac{4 k_\gamma^2}{4} \right) = - \epsilon_1 \epsilon_2 \epsilon_3 \tau_\gamma^2 = \tau_\gamma^2\,.
\]

According to \eqref{tritension-2-dim-sueface}, we  see that the condition of triharmonicity is equivalent to the following system of equations:
\begin{equation}\label{eq-System-tri-S}
\begin{cases}
 k_\gamma^{(3)} k_\gamma + 2  k_\gamma^3  k_\gamma' + 2 k_\gamma'k_\gamma''= 0;\\[1.5ex]
 (\epsilon_1 k_\gamma''-2 \epsilon_2 k_\gamma^3 ) \tau_\gamma^2 +k_\gamma^{(4)}(s) + 10 k_\gamma^2 k_\gamma'' + 15 k_\gamma (k_\gamma ')^2+ k_\gamma^5 = 0.
\end{cases} 
\end{equation}
Since we look for curves $\gamma$ which are \textit{proper} triharmonic, we can assume $k_\gamma > 0$. Then, multiplying  the first equation of \eqref{eq-System-tri-S} by $k_\gamma$, we deduce
\[
\big( k_\gamma'' k_\gamma^2 \big)' + \dfrac{2}{5}(k_\gamma^5)' = 0 \,.
\]
Now integration yields 
$$5k_\gamma'' k_\gamma^2 + 2k_\gamma^5 = c_1
$$ 
for some constant $c_1\in\R$. Since we are interested in non constant solutions, we also assume that $k'_\gamma \neq 0$. Then, after multiplying by $2 k_\gamma' k_\gamma^{-2}$, we deduce that:
\[
\big(5(k_\gamma')^2\big)' + (k_\gamma^4)' + \left( \dfrac{2c_1}{k_\gamma} \right)' = 0\,.
\]
Integrating again, we obtain 
\begin{equation}\label{eq-ODE}
5(k_\gamma')^2 +  k_\gamma^4 +  \dfrac{2c_1}{k_\gamma} = c_2 
\end{equation}
for some constant $c_2\in\R$. Now, in order to reduce technicalities, we can assume $c_1 = 0$, $c_2 = 1$ and denote by $\overline{k}_\gamma(s)$ a non constant positive solution of \eqref{eq-ODE} defined on an interval around $s=s_0$ with initial condition $\overline{k}(s_0)<(1/2)^{1/4}$.
Taking derivatives we easily obtain 
\[
10 \overline{k}_\gamma' \overline{k}_\gamma'' + 4 \overline{k}_\gamma^3 \overline{k}_\gamma' = 0
\]
 and, due the fact that $\overline{k}_\gamma' \neq 0$ and $\overline{k}_\gamma'' = -\dfrac{2}{5}\overline{k}_\gamma^3$, we have
\[
\epsilon_1 \overline{k}_\gamma'' -2\epsilon_2 \overline{k}_\gamma^3 = - \frac{2}{5} \left( \epsilon_1 + 5\epsilon_2\right)\overline{k}_\gamma^3 \neq 0 \,.
\]
Therefore, after a straightforward manipulation, we find that the second equation of the triharmonicity system \eqref{eq-System-tri-S} becomes
\begin{equation}\label{eq:taugammainkbar}
{\tau}^2_\gamma(s)= \frac{63(1-2 \overline{k}_\gamma^4)}{10 \overline{k}_\gamma^2(\epsilon_1+5\epsilon_2)}\,.
\end{equation}
Thus, choosing $\epsilon_1=-1$ and $\epsilon_2=1$ there exists an open neighborhood of $s_0$ on which the right-hand side of \eqref{eq:taugammainkbar} is positive. This ensures  the existence of a function $\overline{\tau}_\gamma(s)$ defined in an open neighborhood of $s=s_0$ and such that \eqref{eq-System-tri-S} is verified. Then the conclusion of Theorem~\ref{Th-trihar-curve-k-nonconstant} follows immediately from the Fundamental Existence Theorem~\ref{Th-Fund-curves}. In details, given $p_0\in \R^3_1$ and the two non constant functions $\overline{k}_\gamma(s)$ and $\overline{\tau}_\gamma(s)$ as described above, we can choose $\{T_0,N_0,B_0\}$ such that $\langle T_0,T_0\rangle=\epsilon_1=-1$, $\langle N_0,N_0\rangle=\epsilon_2=1$ and  $\langle B_0,B_0\rangle=\epsilon_3=1$. From Theorem~\ref{Th-Fund-curves} there exists a curve $\gamma(s)$ parametrized by arc length with Frenet Frame field coinciding with $\{T_0,N_0,B_0\}$ in $s_0$ and curvature functions coinciding with the given functions $\overline{k}_\gamma(s)$ and $\overline{\tau}_\gamma(s)$. Then, for a suitable open neighborhood $I$ of $s_0$ and $\delta>0$ sufficiently small the surface $S_{\gamma}$ is Lorentzian and the curve $\gamma(s)=X(s,0)$ is a triharmonic curve into  $S_{\gamma}$ with non constant curvature $k(s)=\overline{k}_\gamma(s)$.

\section{$r$-harmonic $3$-Frenet curves} \label{Sec-3-Frenet-curves}
In this section we shall prove Theorem~\ref{Th-r-harm-cond-3-dim-space-form} and his corollaries.

\begin{proof}[Proof of Theorem~\ref{Th-r-harm-cond-3-dim-space-form}]
First, it is important for us to establish a version in the semi-Riemannian  context of Lemma\link2.1 of \cite{MR4542687}. Indeed, we prove:

\begin{lemma}\label{lemma-Branding-nablaelle-pseudo}
Let $\gamma(s)$ be a $3$-Frenet helix into a semi-Riemannian manifold $(M^m_t, \langle,\rangle)$, $m \geq 3, 1 \leq t \leq m-1$. Then, for all $\ell \geq 0$, we have:
\begin{equation}\label{formula-nablaelleTT-3-Frenet-pseudo}
\begin{cases}
\nabla_T^{2\ell} T= {\mathcal A}_\ell \, T + {\mathcal B}_\ell\,B \\
\nabla_T^{2\ell+1} T={\mathcal C}_\ell\,N\,,
\end{cases} 
\end{equation}
where
\begin{equation}\label{eq-abc}
\left \{
\begin{array}{llll}
{\mathcal A}_0&=& 1 &\\
{\mathcal A}_\ell&=&(-1)^{\ell} \epsilon_1 \,\epsilon_2^{\ell} \kappa^2 (\epsilon_1\kappa^2 + \epsilon_3 \tau^2)^{\ell - 1} & {\rm if} \,\, \ell \geq 1\\
{\mathcal B}_0&=& 0 &\\
{\mathcal B}_\ell&=&-(-1)^{\ell} \epsilon_3\,\epsilon_2^{\ell} \kappa\,\tau (\epsilon_1\kappa^2 + \epsilon_3 \tau^2)^{\ell - 1} & {\rm if} \,\, \ell \geq 1\\
{\mathcal C}_\ell&=& (-1)^{\ell}\,\epsilon_2^{\ell + 1}\, \kappa (\epsilon_1 \, \kappa^2 + \epsilon_3 \, \tau^2)^\ell& {\rm if} \,\, \ell \geq 0\,.
\end{array} \right .
\end{equation}
\end{lemma}
\begin{proof} The proof of this lemma can be easily carried out separating the two cases $\nabla_T^{2\ell} $ and  $\nabla_T^{2\ell+1} $. In both cases a simple induction argument which uses the Frenet equations enables us to obtain the required result.
\end{proof}
As an application, we can now compute explicitly the condition of $r$-harmoni\-city for $3$-Frenet helices in any semi-Riemannian space form.

Using the expression \eqref{tensor-curvature-N(c)} for the Riemannian curvature tensor field $R$ we can rewrite \eqref{r-harmonicity-curves} as follows:
\begin{equation}\label{r-harmonicity-curves-space-form}
\tau_r(\gamma)=\nabla^{2r-1}_T T+ c\,\sum_{\ell=0}^{r-2}(-1)^\ell \big(\langle \nabla^{\ell}_T T,T \rangle \nabla^{2r-3-\ell}_T T -\langle \nabla^{2r-3-\ell}_T T,T \rangle \nabla^{\ell}_T T\big) =0 \,.
\end{equation}
In the case that $r = 2$, using \eqref{formula-nablaelleTT-3-Frenet-pseudo} we can easily compute the bitension field of $\gamma$:
\begin{eqnarray*}
\tau_2(\gamma) & = &\nabla^3_T T + c \big( \langle T,T \rangle \nabla_T T - \langle \nabla_T T,T \rangle T \big) \\
&=&  - \,\epsilon_2^2\, \kappa (\epsilon_1 \, \kappa^2 + \epsilon_3 \, \tau^2) \,N + c \, \epsilon_1 \epsilon_2 \, \kappa \, N\\
& = & - \,\epsilon_2\, \kappa \big( \epsilon_2 \, (\epsilon_1 \, k^2 + \epsilon_3 \, \tau^2)- c \epsilon_1 \big) \, N 
\end{eqnarray*}
from which (i) follows immediately.

Now, let $r=3$. Then the expression \eqref{r-harmonicity-curves-space-form} becomes:
\[
\begin{aligned}
\tau_3(\gamma) & = \nabla^5_T T + c \big( \langle T,T \rangle \nabla_T^3 T - \langle \nabla^3_T T,T \rangle T - \langle \nabla_T T,T \rangle \nabla^2_T T + \langle \nabla^2_T T,T \rangle \nabla_T T \big) \\
& =   \epsilon_2^3\, \kappa (\epsilon_1 \, \kappa^2 + \epsilon_3 \, \tau^2)^2\,N - c \, \epsilon_1 \,\epsilon_2^2\, \kappa (\epsilon_1 \, \kappa^2 + \epsilon_3 \, \tau^2) \, N - c \, \epsilon_2^2 \, k^3 N\\
& =  - \, \, \kappa \Big[ - \epsilon_2 \, (\epsilon_1 \, \kappa^2 + \epsilon_3 \, \tau^2)^2 + c \, ( \epsilon_1\,\epsilon_3 \, \tau^2\, + 2 \kappa^2 ) \Big] \, N 
\end{aligned}
\]
from which (ii) follows immediately.

Now we deal with the general case $r \geq 4$. First, we suppose that $r=2s$ is even. Using \eqref{tensor-curvature-N(c)} into \eqref{r-harmonicity-curves} and separating even and odd indices we obtain:
\[
\begin{aligned}
\tau_{2s}(\gamma)  = &\; {\mathcal C}_{2s-1}\,N \\
&+c\Big [\sum_{j=0}^{s-1}  \langle \nabla_T^{2j} T,T \rangle\nabla_T^{2(2s-j-2)+1} T -\langle \nabla_T^{2(2s-j-2)+1} T,T \rangle \nabla_T^{2j} T \Big]\\
&  -c\Big [\sum_{j=1}^{s-1}  \langle \nabla_T^{2(j-1)+1} T,T \rangle\nabla_T^{2(2s-j-1)} T -\langle \nabla_T^{2(2s-j-1)} T,T \rangle \nabla_T^{2(j-1)+1} T \Big].
\end{aligned}
\] 
Next, using Lemma~\ref{lemma-Branding-nablaelle-pseudo} we find:
\begin{eqnarray*}
\tau_{2s}(\gamma) & = &{\mathcal C}_{2s-1}\,N +c\sum_{j=0}^{s-1}  \langle \nabla_T^{2j} T,T \rangle\nabla_T^{2(2s-j-2)+1} T \\
&  &+c\,\sum_{j=1}^{s-1} \langle \nabla_T^{2(2s-j-1)} T,T \rangle \nabla_T^{2(j-1)+1} T \\
& =& \Big [{\mathcal C}_{2s-1}+c\epsilon_1\,\sum_{j=0}^{s-1}{\mathcal A}_{j}{\mathcal C}_{2s-j-2}+ c\epsilon_1\,\sum_{j=1}^{s-1}{\mathcal C}_{j-1}{\mathcal A}_{2s-1-j}\Big ] N\,.
\end{eqnarray*} 
Finally, using the explicit expressions for ${\mathcal A}_{\ell},\,{\mathcal C}_{\ell}$ given in \eqref{eq-abc} and performing a straightforward simplification we find
\begin{equation*}
\tau_{2s}(\gamma)= \epsilon_2 \kappa \Big [(\epsilon_1 \kappa^2 + \epsilon_3 \tau^2)^{2s-3}\Big(- \epsilon_2  (\epsilon_1  \kappa^2 + \epsilon_3  \tau^2)^2 + c \big [  (2s - 1) \kappa^2 +\epsilon_1 \epsilon_3 \tau^2 \big ]\Big) \Big ] N
\end{equation*}
from which the conclusion (iii) follows immediately. The case that $r$ is odd is similar and so we omit further details.
\end{proof}

\begin{proof}[Proof of Corollary~\ref{Cor-facile}]
Since $c = 0$, we deduce from (\ref{r-harmonicity-curves-space-form}) that $\gamma$ is proper $r$-harmonic if and only if $\epsilon_1 \, \kappa^2 + \epsilon_3 \tau^2 = 0$. Therefore, there exist solutions only if $\epsilon_1$ and $\epsilon_3$ have opposite sign. Since the index of the metric is $1$, it follows that $\epsilon_2 = 1$ and so the normal vector field $N$ is space-like. Finally, the $r$-harmonicity condition clearly reduces to $\kappa^2 = \tau^2$.
\end{proof}

\begin{proof}[Proof of Corollary~\ref{Cor-tecnico}]
Essentially, the proof consists in the analysis of \eqref{eq-r-harmonicity-spaceforms-r>3}.

\textbf{Case} (i) -- By assumption $\epsilon_2=1$. Then $\epsilon_1 \epsilon_3= -1$. It follows immediately that if $\kappa^2 = \tau^2$ then the curve $\gamma$ is $r$-harmonic for all $r > 3$.\\
Next, we study the equation $- \epsilon_2 \, (\epsilon_1 \, \kappa^2 + \epsilon_3 \, \tau^2)^2 + c \, \big( \epsilon_1\,\epsilon_3 \, \tau^2\, + (r - 1) \kappa^2\big) = 0$ when $r \geq 3$.
This equation in our case is equivalent to
$$\tau^4  + (  c - 2 \kappa^2 ) \tau^2  +\big ( \kappa^4 -  \, c\,(r - 1) \kappa^2\big ) = 0\,,$$
that is a second order degree equation in $\tau^2$. Now, a routine analysis shows that there exist positive roots if and only if either (a) or (b) holds.

\textbf{Case} (ii) -- It is immediate.

\textbf{Cases} (iii) and (iv) -- The computations are similar to Case (i), so we omit further details.
\end{proof}

\section{$n$-Frenet curves, $n \geq 4$}\label{Sec-n-Frenet}
Starting from Theorem~\ref{biharmonic-n-Frenet-curve}, in the introduction we stated several results concerning $n$-Frenet curves with $n \geq 4$. This section is devoted to their proofs. First we establish:

\begin{proposition}\label{prop-n-Frenet-bi}
Let $\gamma$ be an $n$-Frenet curve parametrized by the arc length $s$ into an $m$-dimensional semi-Riemannian manifold $(M_t^m, g)$, $m \geq n\geq 4$. Then its bitension field is given by:
\begin{equation}\label{eq-tau2-n-general}
\begin{aligned}
\tau_2(\gamma)  = &\; \epsilon_2 \big( - 3 \epsilon_1 \,  k_1 \, k_1' \,  F_1 \, +(k_1'' - \epsilon_1 \, \epsilon_2 \, k_1^3 \, - \epsilon_2 \, \epsilon_3 \, k_1 \, k_2^2) F_2 +\\
&+ \epsilon_3( 2  k_1'  \, k_2 \, + k_1 \,  k_2') \, F_3 + \epsilon_3 \, \epsilon_4  \,k_1 \, k_2 \, k_3 \, F_4 +  \, k_1 \, R(F_2, F_1) F_1  \big)\,.
\end{aligned}
\end{equation}
\end{proposition}

\begin{proof}
We need to compute the first three covariant derivatives of $T$. Using the Frenet equations \eqref{Frenet-field-general-pseudo} we obtain
\begin{small}
\begin{eqnarray*}
\nabla_T T &=& \epsilon_2 \, k_1 \, F_2;\\
\nabla_T^2 T &=& \epsilon_2 \, \nabla_T (k_1 \, F_2) = \epsilon_2 \, \big( k_1' \, F_2 + \, k_1 \, (- \epsilon_1 \, k_1 \, F_1 + \epsilon_3 \, k_2 \, F_3 ) \big)\\
& =& \epsilon_2 ( - \epsilon_1 \, k_1^2 \, F_1 \, + \, k_1' \, F_2 \,  + \epsilon_3 \, k_1 \, k_2 \, F_3 ) ;\\
\nabla_T^3 T &=&  \epsilon_2 ( - \epsilon_1 \, \nabla_T (k_1^2 \,  F_1 )\, + \, \nabla_T ( k_1' \, F_2 )\,  + \epsilon_3 \, \nabla_T ( k_1 \,  k_2 \, F_3) ) \\
& =&  \epsilon_2 ( - \epsilon_1 \, 2 k_1 \, k_1' \,  F_1 \, - \epsilon_1 \, \epsilon_2 \, k_1^3 \, F_2 + k_1'' \, F_2  + k_1' \, (-\epsilon_1 \, k_1 \, F_1\\
&&  + \epsilon_3 \, k_2 \, F_3)  + \epsilon_3 \, ( k_1' \,  k_2 + k_1 \,  k_2') \, F_3 + \epsilon_3\,k_1 \, k_2 (-\epsilon_2 \, k_2 \, F_2 + \epsilon_4 \, k_3 \, F_4))\\
&=& \epsilon_2 \Big( - 3 \epsilon_1 \,  k_1 \, k_1' \,  F_1 \, +(k_1'' - \epsilon_1 \, \epsilon_2 \, k_1^3 \, - \epsilon_2 \, \epsilon_3 \, k_1 \, k_2^2) F_2\\
&& + \epsilon_3( 2  k_1'  \, k_2 \, + k_1 \,  k_2') \, F_3 + \epsilon_3 \, \epsilon_4  \,k_1 \, k_2 \, k_3 \, F_4) \Big).
\end{eqnarray*}
\end{small}
Thus, replacing in \eqref{r-harmonicity-curves} and simplifying we obtain
\begin{eqnarray*}
\tau_2(\gamma) & = & \nabla_T^3 T + R(\nabla_T T, T)T \\
&=& \epsilon_2 \big( - 3 \epsilon_1 \,  k_1 \, k_1' \,  F_1 \, +(k_1'' - \epsilon_1 \, \epsilon_2 \, k_1^3 \, - \epsilon_2 \, \epsilon_3 \, k_1 \, k_2^2) F_2 +\\
&&+ \epsilon_3( 2  k_1'  \, k_2 \, + k_1 \,  k_2') \, F_3 + \epsilon_3 \, \epsilon_4  \,k_1 \, k_2 \, k_3 \, F_4 \big) + \epsilon_2 \, k_1 \, R(F_2, F_1) F_1 
\end{eqnarray*}
and so we have the required expression.
\end{proof}

\begin{proof}[Proof of Theorem~\ref{biharmonic-n-Frenet-curve}]
Since the Frenet frame field $\left \{F_1,\ldots,F_n \right \}$ is an orthonormal frame and \eqref{eq:Rorthogonalfi} holds, we can decompose the bitension field as $\tau_2(\gamma) = \sum\limits_{i=1}^n \epsilon_i \langle \tau_2(\gamma), F_i \rangle \, F_i$.\\

From \eqref{eq-tau2-n-general} and using the fact that $\langle R(F_2,F_1)F_1, F_1\rangle = 0$ we find that
\begin{equation*}
\langle \tau_2(\gamma), F_1 \rangle = -3 \, \epsilon_2 k_1 \, k_1' =0\,.
\end{equation*}
Now, since $\gamma$ is proper biharmonic, we obtain that the first condition in \eqref{eq-biharm-n-Frenet} must hold. Therefore, we can rewrite the bitension field as follows:
\begin{equation*}
\tau_2(\gamma) = \epsilon_2 \big[ ( - \epsilon_1  \epsilon_2  k_1^3  - \epsilon_2  \epsilon_3  k_1  k_2^2) F_2 +  \epsilon_3  k_1   k_2'  F_3 + \epsilon_3 \, \epsilon_4  k_1  k_2  k_3  F_4 +   k_1  R(F_2, F_1) F_1 \big]
\end{equation*}
Now, it is easy to compute all the other components of $\tau_2(\gamma)$:
\begin{eqnarray*}
\langle \tau_2(\gamma), F_2 \rangle &=& \epsilon_2 \, k_1 \, \big( - ( \epsilon_1 \, k_1^2 \, + \, \epsilon_3 \, k_2^2) \, +  \, \langle R(F_2, F_1) F_1, F_2 \rangle \big) \,; \\
\langle \tau_2(\gamma), F_3 \rangle &=&   \epsilon_2 \, k_1 \, \big(  k_2' \, +  \,  \langle R(F_2, F_1) F_1, F_3 \rangle \big) \,;\\
\langle \tau_2(\gamma), F_4 \rangle &= &\epsilon_2 \, k_1 \, \big(  \epsilon_3 \, k_2 \, k_3 \, +  \, \langle R(F_2, F_1) F_1, F_4 \rangle \big) \,;\\
\langle \tau_2(\gamma), F_i \rangle &=& \epsilon_2 \, k_1 \, \langle R(F_2, F_1) F_1, F_i \rangle \,, \qquad 5 \leq i \leq n\,,
\end{eqnarray*}
as required to end the proof of Theorem~\ref{biharmonic-n-Frenet-curve}.

\end{proof}
A computation similar to Proposition~\ref{prop-n-Frenet-bi}, using just \eqref{r-harmonicity-curves} and \eqref{Frenet-field-general-pseudo}, leads us to:
\begin{proposition}\label{prop-n-Frenet-tri}
Let $\gamma$ be an $n$-Frenet helix parametrized by the arc length $s$ into an $m$-dimensional semi-Riemannian manifold $(M_t^m, g)$, $1 \leq t \leq m-1$, $m \geq n\geq 6$. Then its tritension field is given by:
\begin{eqnarray*}
\tau_3(\gamma) & = & - \epsilon_2 \, \kappa_1 \big\{ - \epsilon_2 \big[ \epsilon_2 (\epsilon_1\, \kappa_1^2 \,  + \epsilon_3 \, \kappa_2^2)^2  + \epsilon_3^2 \, \epsilon_4  \, \kappa_2^2 \, \kappa_3^2 \big] F_2 +\\
&&+ \, \epsilon_3 \, \epsilon_4 \,  \kappa_2 \, \kappa_3 \big[\epsilon_2 (\epsilon_1\, \kappa_1^2 \,  + \epsilon_3 \, \kappa_2^2 )+ \epsilon_4(\epsilon_3 \,  \kappa_3^2 + \epsilon_5 \kappa_4^2 )\, \big] F_4 \\
&& -\epsilon_3 \, \epsilon_4 \epsilon_5 \epsilon_6 \,\kappa_2 \, \kappa_3 \, \kappa_4 \, \kappa_5 F_6 +  \epsilon_2 \, (2 \epsilon_1\, \kappa_1^2  + \epsilon_3 \, \kappa_2^2 ) R( F_2, F_1) F_1 \\
&&+ \epsilon_2 \, \epsilon_3 \, \kappa_1 \,  \kappa_2 \, R( F_3 , F_2) F_1 - \epsilon_3 \, \epsilon_4 \,\kappa_2 \, \kappa_3 R(F_4, F_1) F_1 \big\}.
\end{eqnarray*}
\end{proposition}
\begin{proof}
From \eqref{r-harmonicity-curves}, the tritension field can be written as
\begin{equation}\label{tri-bis}
\tau_3(\gamma) = \nabla_T^5 T + R(\nabla_T^3 T, T) T - R(\nabla_T^2 T, \nabla_T T) T\,.
\end{equation}
Thus we need to compute the first five covariant derivatives of $T$:
\begin{eqnarray*}
\nabla_T T &=& \epsilon_2 \, \kappa_1 \, F_2;\\
\nabla_T^2 T &=& \epsilon_2 \, \kappa_1 ( - \epsilon_1 \, \kappa_1 \, F_1 + \epsilon_3 \, \kappa_2 \, F_3 ) ;\\
\nabla_T^3 T &=& - \epsilon_2 \, \kappa_1 \big[  \epsilon_2 (\epsilon_1\, \kappa_1^2 \,  + \epsilon_3 \, \kappa_2^2 ) F_2 - \epsilon_3 \, \epsilon_4 \,\kappa_2 \, \kappa_3 \, F_4\big];\\
\nabla_T^4 T 
&=& - \epsilon_2 \, \kappa_1 \big[ - \epsilon_1 \, \epsilon_2 \, \kappa_1  (\epsilon_1\, \kappa_1^2 \,  + \epsilon_3 \, \kappa_2^2 ) F_1 \\
&&+ \epsilon_3 \, \kappa_2 \big(\epsilon_2 (\epsilon_1\, \kappa_1^2 \,  + \epsilon_3 \, \kappa_2^2 ) + \epsilon_3 \, \epsilon_4  \, \kappa_3^2 \big) F_3 
 - \epsilon_3 \, \epsilon_4 \epsilon_5 \,\kappa_2 \, \kappa_3 \, \kappa_4 \, F_5 \big];\\
\nabla_T^5 T 
&=& - \epsilon_2 \, \kappa_1 \big[ - \epsilon_2 \big( \epsilon_2 (\epsilon_1\, \kappa_1^2 \,  + \epsilon_3 \, \kappa_2^2)^2  + \epsilon_3^2 \, \epsilon_4  \, \kappa_2^2 \, \kappa_3^2 \big) F_2 \\
&&+ \, \epsilon_3 \, \epsilon_4 \,  \kappa_2 \, \kappa_3 \big(\epsilon_2 (\epsilon_1\, \kappa_1^2 \,  + \epsilon_3 \, \kappa_2^2  )+ \epsilon_4 (\epsilon_3 \,  \kappa_3^2 + \epsilon_5 \kappa_4^2 )\, \big) F_4\\
&&  -\epsilon_3 \, \epsilon_4 \epsilon_5 \epsilon_6 \,\kappa_2 \, \kappa_3 \, \kappa_4 \, \kappa_5 F_6\big ].
\end{eqnarray*}
Using these expressions in \eqref{tri-bis} it is straightforward to end the proof.
\end{proof}

Now, we examine more in detail the case that the ambient is a semi-Rie\-mannian space form. As a consequence of Proposition~\ref{prop-n-Frenet-tri}, inspection of the component of the tritension field along $F_6$ suggests that there are no proper solutions if $n \geq 6 $. Therefore, it is geometrically natural to restrict our attention to the case that $m=n=5$. In this case the computations are equivalent to those in Proposition~\ref{prop-n-Frenet-tri} with the additional assumption $\kappa_5=0$. Then, using \eqref{tensor-curvature-N(c)} and simplifying, we easily deduce Theorem~\ref{cor-n-Frenet-tri-spaceforms}.\\

As for Examples~\ref{cor-controes-triharmonic} and \ref{cor-controes-triharmonic-bis} it is easy to check that all the conditions of Theorem~\ref{cor-n-Frenet-tri-spaceforms} are verified with $c=1$. Then the conclusion follows from the Fundamental Existence Theorem~\ref{Th-Fund-curves}.

\section{Proof of Proposition~\ref{Th-general-r-harmonicity-3-Frenet-helices} and Theorem~\ref{Th-r-harm-eliche-product-result1}}\label{sec-Lorentzian-product} 
\begin{proof}[Proof of Proposition~\ref{Th-general-r-harmonicity-3-Frenet-helices}] The assumption \eqref{eq-r-harm-general-helices-r>=4}, together with the expressions given in Lemma~\ref{lemma-Branding-nablaelle-pseudo}, enable us to conclude that
\[
\nabla_T^j T=0 
\]
for all $j \geq 3$. Then direct inspection of the tension field \eqref{r-harmonicity-curves} shows that all its terms identically vanish when $r\geq 4$.
\end{proof}
The proof of Theorem~\ref{Th-r-harm-eliche-product-result1} is more laborious and so we carry out some preliminary work. 

First, we observe that in the ambient space $\left ( \R \times N^{m-1}(c),g_{{\rm prod}}\right )$ the curvature terms of the $r$-tension field  \eqref{r-harmonicity-curves} are always tangent to the factor $N^{m-1}(c)$.
It follows readily from Lemma~\ref{lemma-Branding-nablaelle-pseudo} that, if \eqref{eq-r-harm-general-helices-r>=4} does \textit{not} hold, then a \textit{necessary condition} for a helix $\gamma$ to be proper $r$-harmonic, $r \geq 2$, is the fact that its normal vector field $N$ is itself tangent to $N^{m-1}(c)$.

This can be achieved only if $\epsilon_2=1$ and the component of $\gamma$ along the factor $\R$ is a linear function of $s$. Therefore, without loss of generality since vertical translations are isometries, in the family \eqref{eq-def-gamma} we have assumed that this component of $\gamma$ is equal to $d_1\,s$. Moreover, we observe that all the conclusions of Theorem~\ref{Th-r-harm-eliche-product-result1} are trivial if $d_1=0$. Therefore, from now on, we shall assume $d_1 \neq 0$. To fix notation, we write as follows the Frenet equations verified by $\alpha$ in $N^{m-1}(c)$:
\begin{equation}\label{Fren-eq-pseudo-alpha}
\begin{cases}
\widetilde{\nabla}_{T_\alpha} T_\alpha= \kappa_\alpha N_\alpha \\
\widetilde{\nabla}_{T_\alpha} N_\alpha=- \kappa_\alpha T_\alpha+\tau_\alpha B_\alpha\\
\widetilde{\nabla}_{T_\alpha} B_\alpha=-\tau_\alpha N_\alpha\,,
\end{cases}
\end{equation}
where here $\widetilde{\nabla}$ denotes the covariant derivative in $N^{m-1}(c)$. Also, we write the Frenet field of $\gamma$ as follows:
\[
T=\dot{\gamma}=d_1\, \partial_t+d_2 \,\dot{\alpha}(d_2s)= d_1\, \partial_t+d_2\, T_\alpha\,, \quad N=\widetilde{N}\,, \quad B=-B_t \,\partial_t + \widetilde{B}\,,
\]
where $\widetilde{N},\widetilde{B}$ are tangent to the factor $N^{m-1}(c)$. Note that
\[
\epsilon_1=\langle T,T \rangle= - d_1^2 +d_2^2
\]
forces the assumptions \eqref{eq-d} on the constants $d_1,d_2$.
\begin{proof}[Proof of Theorem~\ref{Th-r-harm-eliche-product-result1}]
First, we establish
\begin{lemma}\label{lemma-gamma-elica} The curve $\gamma$ verifies the Frenet equations
\begin{equation}\label{Fren-eq-pseudo-bis}
\begin{cases}
\nabla_T T=\epsilon_2 \kappa N \\
\nabla_T N=-\epsilon_1 \kappa T+\epsilon_3 \tau B\\
\nabla_T B=-\epsilon_2 \tau N\,,
\end{cases}
\end{equation}
with $\epsilon_2=1$ and
\begin{equation}\label{eq-kappaetauproduct}
\kappa=(\epsilon_1+d_1^2)\kappa_\alpha\,, \quad \quad \tau=\sqrt{\epsilon_3\,(\epsilon_1+d_1^2) \left (\tau_\alpha^2-\epsilon_1 \,d_1^2\,\kappa_\alpha^2 \right )}\,.
\end{equation}
Moreover, $\gamma$ does not verify \eqref{eq-r-harm-general-helices-r>=4}.
\end{lemma}
\begin{proof} We compute
\[
\nabla_T T= \nabla_{ d_1\, \partial_t+d_2\, T_\alpha} \left ( d_1\, \partial_t+d_2\, T_\alpha\right )=d_2^2 \widetilde{\nabla}_{T_\alpha} T_{\alpha}= d_2^2 \kappa_\alpha N_\alpha
\]
from which we deduce $N=N_\alpha$, $\epsilon_2=1$ and $\kappa$ as in \eqref{eq-kappaetauproduct}. Similarly
\[
\nabla_T N=d_2 \widetilde{\nabla}_{T_\alpha} N_\alpha=-d_2\, \kappa_\alpha T_\alpha + d_2\, \tau_\alpha B_\alpha\,.
\]
Next, equalling the right-hand side to $-\epsilon_1 \kappa T + \epsilon_3 \tau B$ and considering the components tangent to $\R$ and $N^{m-1}(c)$ respectively we find:
\begin{equation*}\label{eq-1-lemma}
\begin{array}{ll}
{\rm (i)}& \epsilon_1 \kappa d_1 + \epsilon_3 \tau B_t=0 \\
{\rm (ii)}& \big (\epsilon_1 \kappa d_2-\kappa/d_2 \big )T_\alpha + d_2 \tau_\alpha B_\alpha-\epsilon_3 \tau \widetilde{B}=0
\end{array}
\end{equation*}
which imply
\begin{equation}\label{eq-2-lemma}
\begin{array}{ll}
{\rm (i)}& B_t=-\epsilon_1 \epsilon_3 (\kappa / \tau) d_1 \\
{\rm (ii)}& \widetilde{B}=(\epsilon_3 /  \tau)\Big (\big (\epsilon_1 \kappa d_2-\kappa/d_2 \big )T_\alpha + d_2\tau_\alpha B_\alpha\Big )
\end{array}
\end{equation}
We use these facts to compute the expression of $\tau$:
\begin{eqnarray}\nonumber
\tau&=& \langle \nabla_T N,B \rangle \\\nonumber
&=&\langle -d_2\, \kappa_\alpha T_\alpha + d_2\, \tau_\alpha B_\alpha,-B_t \,\partial_t + \widetilde{B} \rangle\\
&=&\langle -d_2\, \kappa_\alpha T_\alpha + d_2\, \tau_\alpha B_\alpha,\frac{\epsilon_3}{\tau} \Big (\big (\epsilon_1 \kappa d_2-\kappa/d_2 \big )T_\alpha + d_2 \tau_\alpha B_\alpha\Big )\rangle\\\nonumber
&=&\frac{\epsilon_3}{\tau} \Big (d_2^2 \tau^2_\alpha -\kappa^2(\epsilon_1-\frac{1}{d_2^2} \Big )\,.
\end{eqnarray}
Next, using $d_2^2=\epsilon_1+d_1^2$ and the assumption \eqref{eq-d-condizione} we easily deduce the expression for $\tau$ given in \eqref{eq-kappaetauproduct}.

Finally, a simple computation yields
\[
\epsilon_1 \kappa^2 +\epsilon_3 \tau^2 =\left(d_1^2+\epsilon_1 \right)\left(\kappa^2_\alpha +\tau^2_\alpha \right )\neq 0\,.
\]
Thus $\gamma$ does not verify condition \eqref{eq-r-harm-general-helices-r>=4}, as required to end the lemma.
\end{proof}
To deal with the curvature terms of the $r$-tension field $\tau_r(\gamma)$ it is useful to compute some scalar products, denoted $\langle,\rangle_g$, with respect to the metric $g$ of $N^{m-1}(c)$. 

More precisely, we make the following slight abuse of notation: writing vectors as $X=-X_t\partial_t+\widetilde{X},Y=-Y_t\partial_t+\widetilde{Y}$, we shall consider
\[
\langle X, Y\rangle_g=\langle \widetilde{X}, \widetilde{Y}\rangle_g \,.
\]We have:
\begin{equation}\label{eq3}
\begin{array}{l}
\langle T, T\rangle_g =\epsilon_1+d_1^2\\
\langle T, N \rangle_g=0\\
\langle T, B\rangle_g =\epsilon_1 \epsilon_3 d_1^2 \frac{\kappa}{\tau}
\end{array}
\end{equation}
where, for the third scalar product, we have used \eqref{eq-2-lemma}(i) and 
\[
0=\langle T, B\rangle=d_1 B_t+\langle T, B\rangle_g \,.
\]

We also point out that in our product space, with a notation as above,
\[
R(X,Y)Z=c \Big (\langle \widetilde{Y}, \widetilde{Z}\rangle_g \widetilde{X}-\langle \widetilde{X}, \widetilde{Z}\rangle_g \widetilde{Y}\Big )\,.
\]
Now, we can proceed to the computation of $\tau_r(\gamma)$ as given in \eqref{r-harmonicity-curves}. 

First we assume that $r$ is even, say $r=2s$. We use Lemma~\ref{lemma-Branding-nablaelle-pseudo} and split the sum into even and odd values for the index $\ell$. Computing in a fashion similar to the proof of Theorem~\ref{Th-r-harm-cond-3-dim-space-form} we manage to express the $2s$-tension field using the functions ${\mathcal A}_{\ell},{\mathcal B}_{\ell},{\mathcal C}_{\ell}$ introduced in Lemma~\ref{lemma-Branding-nablaelle-pseudo}. 
More precisely, we obtain:
\begin{eqnarray}
\tau_{2s}(\gamma)&=&{\mathcal C}_{2s-1} N+c \sum_{j=0}^{s-1}{\mathcal C}_{2s-j-2}\big ( \langle T, T\rangle_g {\mathcal A}_{j}+\langle T, B\rangle_g {\mathcal B}_{j}\big )N\\\nonumber
&&+c \sum_{j=1}^{s-1}{\mathcal C}_{j-1}\big ( \langle T, T\rangle_g {\mathcal A}_{2s-1-j}+\langle T, B\rangle_g {\mathcal B}_{2s-1-j} \big )N\\\nonumber
&=&\Big[-(-1)^{(2 s)} \kappa (\kappa^2 \epsilon_1+\epsilon_3 \tau^2)^{(-3+2 s)} \Big(\kappa^4+2 \kappa^2 \epsilon_1 \epsilon_3 \tau^2+\tau^4\\\nonumber
&&\,\,\,\,\,-c (\kappa^2 (-1+2 s+d_1^2 \epsilon_1)+(d_1^2+\epsilon_1) \epsilon_3 \tau^2)\Big) \Big ] N
\end{eqnarray}
where for the various simplifications we have also used \eqref{eq3}. We have proved in Lemma~\ref{lemma-gamma-elica} that in our family of curves condition \eqref{eq-r-harm-general-helices-r>=4} never holds. Therefore the condition of proper $r$-harmonicity, when $r$ is even, is equivalent to
\begin{equation}\label{r-harm-product}
\kappa^4+2 \kappa^2 \epsilon_1 \epsilon_3 \tau^2+\tau^4-c (\kappa^2 (r-1+d_1^2 \epsilon_1)+(d_1^2+\epsilon_1) \epsilon_3 \tau^2)=0\,.
\end{equation}
A similar computation shows that the same condition rules $r$-harmonicty also when $r$ is odd. Finally, we replace the explicit expressions \eqref{eq-kappaetauproduct} for $\kappa$ and $\tau$ into \eqref{r-harm-product}. Then we see that the $r$-harmonicity equation \eqref{r-harm-product} becomes equivalent to
\[
(\kappa_\alpha^2+\tau_\alpha^2)^2-c ((r-1)\kappa_\alpha^2 +\tau_\alpha^2)=0\,.
\]
But the latter is precisely the condition of $r$-harmonicity for the $3$-Frenet helix $\alpha$ in $N^{m-1}(c)$ (see \cite{MR4542687}), as required to end the proof of Theorem~\ref{Th-r-harm-eliche-product-result1}.
\end{proof}

\section{$r$-harmonic curves in Robertson-Walker space time}\label{Sec-Rob-Walker}
In this section we study $r$-harmonic curves into the Robertson-Walker Lorentzian manifold $\mathcal{RW}^{m}_1 =\left ( J \times N^{m-1}(c),g_f\right )$, where the metric $g_f$ was introduced in \eqref{eq-gf-walker}.
In particular, we shall prove Theorems~\ref{biharmonic-cond-Robertson-Walker} and \ref{Th-2-RW}.

A vector field $X$ on $\mathcal{RW}^{m}_1$ can be decomposed as
\[
X =  (-X_t,\widetilde{X}) = - X_t \partial_t + \widetilde{X} ,
\]
where $\widetilde{X}$ is a vector field tangent to $N^{m-1}(c)$ and $X_t =g_f(X,\partial_t)= \langle X, \partial_t\rangle$.

Preliminarily, we observe that the action of the Levi-Civita connection in a Robertson-Walker semi-Riemannian manifold is described in the following proposition.

\begin{proposition}[see \cite{Neill}]\label{connection-Robertson-Walker-tangent}
Let $\widetilde{X}, \widetilde{Y}$ be two vector fields tangent to $N^{m-1}(c)$ and $ \partial_t$ a unit vector field tangent to $J$. Then the Levi-Civita connection $\nabla$ of the Robertson-Walker space-time $\mathcal{RW}^{m}_1$ verifies the following properties:
\begin{itemize}
\item[(i)] $\nabla_{\partial_t} \partial_t = 0 \,;$
\item[(ii)] $\nabla_{\partial_t} \widetilde{X} = \nabla_{\widetilde{X}} \partial_t = \dfrac{f'}{f} \widetilde{X}\,;$
\item[(iii)] $\nabla_{\widetilde{X}} \widetilde{Y} = \widetilde{\nabla}_{\widetilde{X}} \widetilde{Y} + \langle \widetilde{X}, \widetilde{Y} \rangle \dfrac{f'}{f} \partial_t$, where $\widetilde{\nabla}$ is the Levi-Civita connection of $N^{m-1}(c)$ and $\langle \widetilde{X}, \widetilde{Y} \rangle=g_f(\widetilde{X},\widetilde{Y})=f^2 g(\widetilde{X},\widetilde{Y})$.
\end{itemize}
\end{proposition}
Next, following \cite{Neill} again, we recall the following properties of the Riemannian curvature tensor of $\mathcal{RW}^{m}_1$:
\begin{proposition}\label{tensor-curvature-Robertson-Walker-tangent}
Let $\widetilde{X}$, $\widetilde{Y}$ and $\widetilde{Z}$ be vector fields tangent to $N^{m-1}(c)$ and $\partial_t$ a unit vector field tangent to $J$. Then the Riemannian curvature operator $R$ of the Robertson-Walker space-time $\mathcal{RW}^{m}_1$ satisfies 
\begin{itemize}
\item[(i)] $R(\widetilde{X}, \widetilde{Y}) \widetilde{Z} = \left( \dfrac{(f')^2 + c}{f^2} \right)\big( \langle \widetilde{Y}, \widetilde{Z}\rangle \widetilde{X}  - \langle \widetilde{X}, \widetilde{Z}\rangle \widetilde{Y} \big)$;
\item[(ii)] $R(\widetilde{X}, \partial_t) \partial_t  = - \left( \dfrac{f''}{f} \right) \widetilde{X}$;
\item[(iii)] $R(\widetilde{X}, \widetilde{Y}) \partial_t = 0$;
\item[(iv)] $R(\widetilde{X}, \partial_t) \widetilde{Y} = - \left( \dfrac{f''}{f} \right) \langle \widetilde{X}, \widetilde{Y}\rangle \partial_t$.
\end{itemize}
\end{proposition}

Our first result in this context is:

\begin{proposition}\label{tension-field-Robertson-Walker-Npartial}
Let $\gamma$ be a $2$-Frenet helix parametrized by the arc length $s$ into an $m$-dimensional Robertson-Walker space-time $\mathcal{RW}^{m}_1$ such that its normal vector field $N$ is equal to $\pm \partial_t$. Then 
$$\tau_r(\gamma) = - \kappa^{2r -3} \left( \kappa^2 + (r - 1)\frac{f''}{f} \right)N\,.
$$
\end{proposition} 

\begin{proof}
We assume $N = \partial_t$, $\epsilon_2 = - 1$ (the case $N = -\partial_t$, $\epsilon_2 = - 1$ is analogous). This implies that $T$ must be space-like and $\epsilon_1 = 1$. Moreover, $T$ is tangent to $N^{m-1}(c)$. Thus, using Proposition~\ref{tensor-curvature-Robertson-Walker-tangent}, 
\begin{equation}\label{2-Frenet-tensor-curvature-Robertson-Walker}
R (N , T) T = -R (T , N) T = \left( \frac{f''}{f} \right) \langle T, T\rangle \partial_t = \frac{f''}{f} N \,. 
\end{equation}
Now assume $r=2$. From \eqref{r-harmonicity-curves} and using \eqref{formula-nablaelleTT-2-Frenet-pseudo}, \eqref{2-Frenet-tensor-curvature-Robertson-Walker} we obtain
\[
\tau_2(\gamma) = \nabla_T^3 T + R (\nabla_T T, T) T =  - \kappa^3 N - \kappa R(N, T) T =- \kappa \left(\kappa^2 + \frac{f''}{f} \right) N\,.
\]
In the case that $r=3$, the expression \eqref{r-harmonicity-curves} becomes:
\begin{eqnarray*}
\tau_3(\gamma) & = & \nabla^5_T T + R (\nabla^3_T T, T) T - R (\nabla^2_T T, \nabla_T T) T\\
& = &  - \, \kappa^5\,N - \kappa^3 R(N, T) T + \kappa^3 R (T, N) T= - \kappa^3 \left(\kappa^2 + 2\frac{f''}{f} \right) N
\end{eqnarray*}
Now we find the general expression of the $r$-tension field for $r \geq 4$. Assume $r=2s$ is even. Then, separating even and odd indices in \eqref{r-harmonicity-curves} and using \eqref{formula-nablaelleTT-2-Frenet-pseudo} and \eqref{2-Frenet-tensor-curvature-Robertson-Walker}:
\[
\begin{aligned}
\tau_{2s}(\gamma)   = &\nabla_T^{2(2s -1) +1} T + \sum_{j=0}^{s-1} R\left (\nabla_T^{2(2s - j -2) + 1} T ,\nabla_T^{2j} T\right )T\\
&- \sum_{j=1}^{s-1}R\left (\nabla_T^{2(2s - j - 1)} T ,\nabla_T^{2(j-1)+1} T\right) T\\
 = &- \kappa^{2(2s -1) +1}N - \sum_{j=0}^{s-1} \kappa^{2(2s - 2) + 1}R\left( N, T \right)T + \sum_{j=1}^{s-1}  \kappa^{2(2s - 2) + 1} R\left( T, N\right) T \\
=& - \kappa^{2(2s -2) + 1} \left( \kappa^2 + (2s - 1)\frac{f''}{f} \right)N\,.
\end{aligned}
\]
The case that $r$ is odd follows the same argument.
\end{proof}

\begin{proof}[Proof of Theorem~\ref{biharmonic-cond-Robertson-Walker}]

The unit tangent vector field of $\gamma$ is given by $\gamma' = T = (0, d\,\alpha')$. We observe that $T$ is tangent to $N^{m-1}(c)$ and $\epsilon_1 = 1$. From the first Frenet equation and Proposition~\ref{connection-Robertson-Walker-tangent} we deduce
\begin{equation}\label{Frenet-Equation-Ttangent}
\epsilon_2 k(s) \, N = \nabla_T T = d^2 \widetilde{\nabla}_{\alpha'} \alpha' +  \frac{f'(t_0)}{f(t_0)} \partial_t\,.
\end{equation}

Now, since $\alpha$ is a geodesic on $N^{m-1}(c)$, $\widetilde{\nabla}_{\alpha'} \alpha' = 0$. Then \eqref{Frenet-Equation-Ttangent} implies $N = \pm \partial_t$, $\epsilon_2 = -1$ and 
$$
k(s)^2 = \kappa^2 = \left( \frac{f'(t_0)}{f(t_0)} \right)^2\,.
$$
Combining again Frenet equations \eqref{Frenet-field-general-pseudo} and Proposition~\ref{connection-Robertson-Walker-tangent} we deduce:
\[
-\kappa T + \epsilon_3 \tau(s) B = \nabla_T N = -\left|\frac{f'(t_0)}{f(t_0)}\right| T\,.
\]
From this equation we obtain $\tau = 0$. This means that $\gamma$ is a $2$-Frenet helix.\\
The curve $\gamma$ is a geodesic if and only if $\kappa=0$, that is $f'(t_0)=0$. Thus, from Proposition~\ref{tension-field-Robertson-Walker-Npartial}, $\gamma$ is proper $r$-harmonic if and only if
\[
f'(t_0) \neq 0\qquad {\rm and} \qquad \kappa^2 + (r-1)\dfrac{f''(t_0)}{f(t_0)} = 0\,.
\]
From this condition it is immediate to conclude the proof.
\end{proof}

\begin{proof}[Proof of Theorem~\ref{Th-2-RW}]
The tangent vector field of $\gamma$ is given by $\gamma' = T = (0, \alpha')$. Thus $T$ is tangent to $N^{m-1}(c)$ and $\epsilon_1 = 1$. From Proposition~\ref{connection-Robertson-Walker-tangent} and the hypothesis $f'(t_0)=0$ we obtain $\nabla_T T =  \widetilde{\nabla}_{\alpha'} \alpha'$.\\
Now, let $\beta(s) = \alpha( f(t_0) s)$. Then $T_\beta = f(t_0)\,\alpha'$ and 
\[
g (T_\beta, T_\beta) = f(t_0)^2 \, g(\alpha', \alpha') = \langle T, T \rangle = 1,
\] 
that is $\beta$ is parametrized by the arc length on $N^{m-1}(c)$. Next,

\begin{equation}\label{cov-derivative-repar}
\nabla_T T =  \widetilde{\nabla}_{\alpha'} \alpha' = \dfrac{1}{f(t_0)^2} \widetilde{\nabla}_{T_\beta}T_\beta
\end{equation}
and by an induction argument
\begin{equation}\label{cov-derivative-repar-general}
\nabla^j_T T = \dfrac{1}{f(t_0)^{j+1}} \widetilde{\nabla}^j_{T_\beta}T_\beta
\end{equation}
for all $j \geq 1$. Let $\widetilde{X}$, $\widetilde{Y}$ and $\widetilde{Z}$ be vector fields tangent to $N^{m-1}(c)$. From Proposition~\ref{tensor-curvature-Robertson-Walker-tangent} and the assumption $f'(t_0)=0$ we deduce that along the curve $\gamma$ we have:
\begin{eqnarray}\label{tensor-curvature-repar}\nonumber
R(\widetilde{X}, \widetilde{Y}) \widetilde{Z}& = &\frac{c}{f(t_0)^2} \big( \langle \widetilde{Y}, \widetilde{Z}\rangle \widetilde{X} - \langle \widetilde{X}, \widetilde{Z}\rangle \widetilde{Y} \big)\\
&=& c \big( g(\widetilde{Y}, \widetilde{Z}) \widetilde{X}  - g( \widetilde{X}, \widetilde{Z}) \widetilde{Y} \big) \\\nonumber
&=& R^{N^{m-1}(c)}(\widetilde{X}, \widetilde{Y}) \widetilde{Z}\,.
\end{eqnarray}
Finally, using \eqref{cov-derivative-repar-general} and \eqref{tensor-curvature-repar} in the expression of $r$-tension field \eqref{r-harmonicity-curves},
\[
\begin{aligned}
&\tau_r(\gamma)=\nabla^{2r-1}_T T+ \sum_{\ell=0}^{r-2}(-1)^\ell R\left (\nabla^{2r-3-\ell}_T T ,\nabla^{\ell}_T T\right )T\\
&= \dfrac{1}{f(t_0)^{2r}} \widetilde{\nabla}^{2r-1}_{T_\beta}T_\beta + \sum_{\ell=0}^{r-2}(-1)^\ell R\left (\dfrac{1}{f(t_0)^{2r-2-\ell}} \widetilde{\nabla}^{2r-3-\ell}_{T_\beta}T_\beta,\dfrac{1}{f(t_0)^{\ell+1}} \widetilde{\nabla}^{\ell}_{T_\beta}T_\beta \right )T\\
&= \dfrac{1}{f(t_0)^{2r}} \widetilde{\nabla}^{2r-1}_{T_\beta}T_\beta + \dfrac{1}{f(t_0)^{2r}}\sum_{\ell=0}^{r-2}(-1)^\ell R^{N^{m-1}(c)}\left ( \widetilde{\nabla}^{2r-3-\ell}_{T_\beta}T_\beta, \widetilde{\nabla}^{\ell}_{T_\beta}T_\beta \right )T_\beta\\
& =  \dfrac{1}{f(t_0)^{2r}} \tau_r(\beta)
\end{aligned}
\]
and so the conclusion of the proof follows readily.
\end{proof}

\appendix
\section{The Fundamental Theorem of Curves}\label{appendix}

\begin{theorem}[The Fundamental Theorem of Curves in Semi-Riemannian Geometry]
Let $M=(M^m_t,\langle, \rangle)$ be an $m$-dimensional semi-Riemannian manifold of index $1 \leq t \leq m-1$. Given smooth, positive curvature functions $k_1(s),\ldots,k_{n-1}(s)$ defined on an open real neighbourhood of $s_0\in \R$, an initial point $p_0\in M$ and an initial orthonormal frame $\{T_0,F_{2,0},\ldots,F_{n,0}\}$ at $p_0$ there exists an $n$-Frenet curve $\gamma(s)$ in $M$ parametrized by the arc length $s$ such that
\begin{itemize}
\item[(i)] $\gamma(s_0)=p_0$;
\item[(ii)] The $n$-Frenet frame field $\{T,F_{2},\ldots,F_{n}\}$ of $\gamma$ coincides with 

\noindent $\{T_0,F_{2,0},\ldots,F_{n,0}\}$ at $s=s_0$;
\item[(iii)] The given functions $k_1(s),\ldots,k_{n-1}(s)$ are the curvatures of $\gamma(s)$.
\end{itemize}
\end{theorem}

\begin{proof} 
We follow the approach of \cite{MR3198740} and \cite{Spivak}. Let $p_0\in M$ and let $( U, \varphi) = (U, x^1,  \ldots, x^m) )$ be a local chart of $M_t^m$ around the point $p_0$.
Consider in the chart $( U, \varphi)$ a given curve $\alpha(s) = (\alpha^1(s), \ldots, \alpha^m(s)) = (x^1 \circ \alpha, \ldots, x^m \circ \alpha)$ and let $\{ F_1^\alpha, \ldots, F_n^\alpha \}$ be the Frenet frame field along $\alpha$. Then, with respect to the coordinate frame field $\left\lbrace\frac{\partial}{\partial x^i} \right\rbrace_{i = 1}^m$ we have
\[
T = \alpha'(s) = \dot{\alpha}^j(s) \dfrac{\partial}{\partial x^j}
\]
and
\[
F_i^\alpha = w_i^j \dfrac{\partial}{\partial x^j},\quad i=1, \dots, n
\] 
 Note that $w_1^j =\dot{\alpha}^j$.
Moreover, 
\[
\nabla_T F_i^\alpha = \big( \dot{w}_i^j + \Gamma_{rs}^j \dot{\alpha}^r w_i^s \big)\dfrac{\partial}{\partial x^j},
\]
where the $\Gamma_{rs}^j$'s are the Christoffel symbols of the Levi-Civita connection in the local chart $( U, \varphi)$.
Using the Frenet Equations \eqref{Frenet-field-general-pseudo} we obtain that the $ w_i^j$'s satisfy the following system with $1\leq j\leq m$:
\begin{equation}\label{Frenet-field-general-pseudo-local}
\begin{cases}
\dot{\alpha}^j(s)=w_1^j(s) \\ 
 \dot{w}_1^j + \Gamma_{rs}^j \dot{\alpha}^r w_1^s= \epsilon_2 \,k_1^\alpha\,w_2^j  \\ 
\dot{w}_i^j + \Gamma_{rs}^j \dot{\alpha}^r w_i^s=-\epsilon_{i-1}\,k_{i-1}^\alpha w_{i-1}^j+\,\epsilon_{i+1}\,k_i^\alpha w_{i+1}^j \,,  \quad 1<i<n\\ 
\dot{w}_n^j + \Gamma_{rs}^j \dot{\alpha}^r w_n^s=-\epsilon_{n-1}\,k_{n-1}^\alpha w_{n-1}^j 
\end{cases} 
\end{equation}
If we put  $A(s) = (\alpha^1(s),  \dots, \alpha^m(s))$ and $W_i(s) = (w_i^1(s),  \dots, w_i^m(s)), 1 \leq i \leq n$, then \eqref{Frenet-field-general-pseudo-local} can be rewritten as follows:
\begin{equation}\label{Frenet-field-general-system-ODE}
\left \{
\begin{array}{ccl}
A'(s)&=&W_1(s)\,, \\[1 ex]
W'_i(s) &=& G_i\big(A(s), W_1(s), ..., W_n(s)\big),\quad  i=1,  \dots, n
\end{array} \right .
\end{equation}
where the $G_i$'s are functions which depend only on the curvatures $k_1^\alpha, \dots k_{n-1}^\alpha$, the values of $\epsilon_1 \dots \epsilon_n$ and the Christoffel symbols. 
By way of summary, for a given curve $\alpha(s)$ in the local chart $( U, \varphi)$, system~\ref{Frenet-field-general-pseudo-local} or, equivalently, system~\ref{Frenet-field-general-system-ODE} holds. 

Now, let us suppose to  have $(n-1)$ positive  differentiable functions 

\noindent $k_1(s), \dots, k_{n-1}(s)$ and an initial orthonormal frame $\{T_0,F_{2,0},\ldots,F_{n,0}\}$ at $p_0$ such that
\begin{equation}\label{eq-initial-orthonormal-frame}
\epsilon_1 = \langle T_0, T_0 \rangle , \qquad \epsilon_2 =  \langle F_{2,0}, F_{2,0} \rangle , \quad \dots \quad \epsilon_n = \langle F_{n,0}, F_{n,0} \rangle .
\end{equation}

We consider the system of differential equations in the variables 

\noindent $C(s), V_1(s), \dots, V_n(s)$:
\begin{equation}\label{Frenet-field-existence-system-ODE}
\left \{
\begin{array}{ccl}
C'(s)&=&V_1(s)\,, \\[1 ex]
V'_i(s) &=& G_i\big(C(s), V_1(s), ..., V_n(s)\big),\  i=1,  \dots, n
\end{array} \right .
\end{equation}
where $G_1, \dots G_n$ are formally the functions defined in \eqref{Frenet-field-general-system-ODE}, but they depend  on the given functions $k_1, \dots k_{n-1}$, $\epsilon_1 \dots \epsilon_n$ and the Christoffel symbols.
If we choose the initial conditions
\[
\big(C(s_0), V_1(s_0), V_2(s_0) \dots, V_n(s_0)\big ) =\big (\varphi(p_0), t_0, f_{2,0},\ldots,f_{n,0}\big ),
\]
where $t_0,f_{2,0},\ldots,f_{n,0} \in \mathbb{R}^m$ are the components, with respect to the coordinate frame field, 
of $T_0, F_{2,0},\ldots,F_{n,0}$ respectively, there exists a unique solution $\big(c(s),v_1(s), \dots, v_n(s)\big)$ of \eqref{Frenet-field-existence-system-ODE}, in a neighbourhood of $s_0$, such that 
\[
\big(c(s_0), v_1(s_0), v_2(s_0) \dots v_n(s_0)\big) = \big(\varphi(p_0), t_0, f_{2,0},\ldots,f_{n,0}\big).
\]
Define $\gamma(s) = \varphi^{-1} \circ c(s)$. Then $\gamma(s)$ is a curve on $M_t^m$ such that $\gamma(s_0) = p_0$.\\
Let us define the vector fields $f_i = v_i^j \dfrac{\partial}{\partial x_j}$, $i=1, \dots, n$,  along $\gamma$. Then the frame field $\{f_1, f_2,\dots, f_n\}$  coincides with $\{ T_0, F_{2,0},\ldots,F_{n,0} \}$ at the point $s_0$ and satisfy \eqref{Frenet-field-existence-system-ODE} by construction. To finish the proof, we have to show that $k_1, \dots, k_{n-1}$ are the curvatures of the curve $\gamma(s)$. For this, it is sufficient to prove that $\{f_1,\dots, f_n\}$ is an orthonormal frame field of the type \eqref{eq-initial-orthonormal-frame} along the curve $\gamma$. If the matrix $\Omega$ defined in Remark~\ref{remark-not-skew} were skew-symmetric, then the conclusion could be rapidly obtained by using the method of \cite[Theorem~6, p. 124]{MR1998826}. 

In our case,  generalizing the method of \cite{MR3198740}, we consider the following auxiliary system of ${n(n+1)}/{2}$ differential equations in the variables $X_{ij}(s), 1 \leq i \leq j \leq n$:
{\small \begin{equation}\label{Frenet-field-ortho-system-ODE}
\left \{
\begin{array}{ccll}
X_{1 \, 1}' &= &2\epsilon_2 k_1 X_{1 \,2}&\\[1 ex]
X_{i \, i}' & = & -2\epsilon_{i-1} k_{i-1} X_{i-1 \, i} + 2\epsilon_{i+1} k_i X_{i \, i+1} & 1 < i < n\\[1 ex]
X_{n \, n}' & = & -2\epsilon_{n-1} k_{n-1} X_{n-1 \,n}&\\[1 ex]
X_{1 \, i}' &= & \epsilon_2 k_1 X_{2 \, i} - \epsilon_{i-1} k_{i-1} X_{1 \, i-1} + \epsilon_{i+1} k_i X_{1 \, i+1} & 1 < i < n\\[1 ex]
X_{1 \, n}'  &=&  \epsilon_2 k_1 X_{2 \, n} - \epsilon_{n-1} k_{n-1} X_{1 \, n-1}&\\[1 ex]
X_{i \, j}'  &=&  - \epsilon_{i-1} k_{i-1} X_{i-1 \, j} + \epsilon_{i+1} k_i X_{i+1 \, j}&\\
&& - \epsilon_{j-1} k_{j-1} X_{i \, j-1} + \epsilon_{j+1} k_j X_{i \, j + 1} &1 < i < j < n\\[1 ex]
X_{i \, n}' &=&  - \epsilon_{i-1} k_{i-1} X_{i-1 \, n} + \epsilon_{i+1} k_i X_{i+1 \, n} &\\[1 ex]
&&- \epsilon_{n-1} k_{n-1} X_{i \, n - 1} & 1 < i < n,
\end{array} \right .
\end{equation}
}
with initial conditions at the point $s=s_0$
\[
X_{ii}(s_0) = \epsilon_i, \, 1 \leq i \leq n; \quad X_{ij}(s_0) = 0, \, i < j.
\]
Since $\{ f_1, \dots, f_n\}$ satisfy \eqref{Frenet-field-existence-system-ODE}, then  $Y_{ij} = \langle f_i, f_j \rangle$, $1 \leq i \leq j \leq n$, is a solution of \eqref{Frenet-field-ortho-system-ODE}. Moreover, since $\{f_1, f_2 \dots, f_n\}$  coincide with the orthonormal frame $\{ T_0, F_{2,0},\ldots,F_{n,0} \}$ at the point $s_0$, then
\[
\begin{aligned}
&Y_{1 \, 1}(s_0) = \langle T_0, T_0 \rangle = \epsilon_1\\
&Y_{i \, i}(s_0) = \langle F_{i0}, F_{i0} \rangle = \epsilon_i \quad 2 \leq i \leq n\\
&Y_{ij}(s_0) = \langle F_{i0}, F_{j0} \rangle = 0 \quad  i < j,
\end{aligned}
\]
that is the solution $\lbrace Y_{ij} \rbrace_{1 \leq i \leq j \leq n}$ satisfy the initial conditions at the point $s_0$.
On the other hand, the ${n(n+1)}/{2}$ constant functions given by
\[ 
Z_{ii} = \epsilon_i, \, 1 \leq i \leq n; \quad Z_{ij} = 0, \, i < j
\]
are also solutions of \eqref{Frenet-field-ortho-system-ODE}. 
Thus, by the uniqueness of the solution, we deduce that $Y_{ij} = Z_{ij}, \, 1 \leq i \leq j \leq n$, that is 
\[
\langle f_i, f_i \rangle = \epsilon_i, \, 1 \leq i \leq n; \quad \langle f_i, f_j \rangle = 0, \, i \neq j.
\]
This means that $\{f_1,\dots,f_n\}$ is an orthonormal frame along $\gamma$, as required to end the proof.
\end{proof}


\end{document}